\documentclass[12pt]{amsart}
\usepackage{amsmath}
\usepackage{amstext}
\usepackage{amsfonts}
\usepackage{amssymb}
\usepackage{amsthm}
\usepackage{amsrefs}

\usepackage{microtype}
\usepackage{color,hyperref}
\definecolor{darkblue}{rgb}{0.0,0.0,0.3}
\hypersetup{colorlinks,breaklinks,
            linkcolor=red,urlcolor=red,
            anchorcolor=red,citecolor=red}

\theoremstyle{plain}
\newtheorem{thm}{Theorem}[section]

\newtheorem{cor}[thm]{Corollary}
\newtheorem{prop}[thm]{Proposition}
\newtheorem{lem}[thm]{Lemma}

\theoremstyle{definition}
\newtheorem{defn}[thm]{Definition}
\newtheorem{example}[thm]{Example}
\newtheorem{rem}[thm]{Remark}

\newtheorem{problem}[thm]{Problem}

\newcommand{\bC}{{\mathbb{C}}}
\newcommand{\bN}{{\mathbb{N}}}

\newcommand{\A}{{\mathcal{A}}}
\newcommand{\B}{{\mathcal{B}}}
\newcommand{\D}{{\mathcal{D}}}

\renewcommand{\H}{{\mathcal{H}}}
\newcommand{\K}{{\mathcal{K}}}

\newcommand{\M}{{\mathcal{M}}}
\renewcommand{\O}{{\mathcal{O}}}

\newcommand{\U}{{\mathcal{U}}}

\renewcommand{\phi}{\varphi}

\newcommand{\fA}{{\mathfrak{A}}}
\newcommand{\fB}{{\mathfrak{B}}}
\newcommand{\fM}{{\mathfrak{M}}}
\newcommand{\fN}{{\mathfrak{N}}}

\newcommand{\qand}{\quad\text{and}\quad}

\newcommand{\AND}{\text{ and }}

\newcommand{\cconv}{\overline{\operatorname{conv}}}
\newcommand{\diag}{\operatorname{diag}}

\newcommand{\ca}{\mathrm{C}^*}

\newcommand{\II}[1]{II$_1$}

\begin{document}
\title[The Schur-Horn problem for normal operators]{The Schur-Horn problem for normal operators}

\author[M. Kennedy]{Matthew Kennedy}
\address{Department of Mathematics and Statistics\\ Carleton University\\
Ottawa, ON \; K1S 5B6 \\Canada}
\email{mkennedy@math.carleton.ca}

\author[P. Skoufranis]{Paul Skoufranis}
\address{Department of Mathematics\\Texas A\&M University\\
College Station, Texas\\77843\\USA}
\email{pskoufra@math.tamu.edu}

\begin{abstract}
We consider the Schur-Horn problem for normal operators in von Neumann algebras, which is the problem of characterizing the possible diagonal values of a given normal operator based on its spectral data. For normal matrices, this problem is well-known to be extremely difficult, and in fact, it remains open for matrices of size greater than $3$. We show that the infinite dimensional version of this problem is more tractable, and establish approximate solutions for normal operators in von Neumann factors of type I$_\infty$, II and III. A key result is an approximation theorem that can be seen as an approximate multivariate analogue of Kadison's Carpenter Theorem.
\end{abstract}

\thanks{First author partially supported by research grant from NSERC (Canada).}
\thanks{Second author partially supported by research grant from the NSF (USA)}
\maketitle

\section{Introduction}

In this paper we investigate the Schur-Horn problem for normal operators in von Neumann algebras. The general version of the Schur-Horn problem is to characterize the possible diagonal values of a given operator based on its spectral data. Since the appropriate notion of diagonal for an operator in a von Neumann algebra is the conditional expectation of the operator onto a maximal abelian self-adjoint subalgebra (MASA), the Schur-Horn problem can be formulated in the following way.

\begin{problem}[Schur-Horn] \label{problem:normal-schur-horn}
Let $T$ be an operator in a von Neumann algebra $\fM$, and let $\A$ be a MASA  in $\fM$ with corresponding conditional expectation $E_\A : \fM \to \A$. Determine the elements of the set
\begin{equation} \label{problem:normal-schur-horn:eq-1}
\D_{\A}(T) := \{ E_\A(U^*TU) \mid U \ \text{a unitary in}\ \fM\}.
\end{equation}
\end{problem}

The classical theorem of Schur \cite{S1923} and Horn \cite{H1954} completely solves the Schur-Horn problem for self-adjoint matrices. A great deal of effort has  gone into finding generalizations of this result for self-adjoint operators in infinite-dimensional von Neumann algebras. We are particularly interested in the body of work inspired by the recent papers of Neumann \cite{N1999} and Arveson and Kadison \cite{AK2006}.

For von Neumann factors of type \II1, the work contained in \cites{AM2007, AM2009, BR2014, DFHS2012} culminates in Ravichandran's Schur-Horn theorem for self-adjoint operators \cite{R2012}, which settles an open problem from \cite{AK2006}. For the state-of-the-art results for von Neumann factors of type I$_\infty$, we direct the reader to the recent work of Kaftal and Weiss \cites{KW2008,KW2010}, Jasper \cite{J2013}, and Bownik and Jasper \cites{BJ2013, BJ2014}.

The Schur-Horn problem for normal operators is seemingly substantially more difficult than the corresponding problem for self-adjoint operators. Indeed, for normal matrices the problem is known to be equivalent to a difficult problem about the geometry of orthostochastic matrices (cf. \cite{A2007}), and little is known about the latter problem for matrices of size greater than $3$ (for matrices of size 3, see \cite{AP1979}).

Surprisingly, the Schur-Horn problem for normal operators is more tractable in infinite dimensions. In this paper, we establish approximate Schur-Horn type theorems for normal operators in infinite dimensional von Neumann algebras. Specifically, we establish a characterization of the norm closure $\overline{\D_\A(N)}^{\left\|\cdot\right\|}$ of the set $\D_\A(N)$ in (\ref{problem:normal-schur-horn:eq-1}) when $N$ is a normal operator (satisfying certain assumptions, depending on the context) in a von Neumann factor $\fM$ of type I$_\infty$, II or III.

It is not difficult to establish (cf. Lemma \ref{lem:spectrum-of-expectation-in-convex-hull}) that a necessary condition for an operator $A$ to belong to $\overline{\D_\A(N)}^{\|\cdot\|}$ is that the spectrum $\sigma(A)$ of $A$ belongs to the closed convex hull $\cconv(\sigma(N))$ of the spectrum of $N$. If $\fM$ is a factor of type I$_\infty$, then this condition turns out to be sufficient when $\A$ is a continuous MASA. When $\A$ is the discrete MASA, then we require an additional assumption on the essential spectrum of $N$ to circumvent the difficulties that arise for matrices.

\begin{thm}
\label{thm:SH-BH-Discrete-Continuous}
Let $\A$ be a MASA in $\B(\H)$ for an infinite-dimensional, separable Hilbert space $\H$ and let $N$ be a normal operator in $\B(\H)$. 
If $\A$ is continuous, then
\[
\overline{\D_{\A}(N)}^{\left\|\cdot\right\|} =  \{A \in \A \mid \sigma(A)  \subseteq \cconv(\sigma_e(N))\}.
\]
If $\A$ is discrete, then the above equality also holds provided that $\sigma(N) \subseteq \cconv(\sigma_e(N))$, where $\sigma_e(N)$ denotes the essential spectrum of $N$.
\end{thm}

When $\fM$ is a II$_1$ factor and $N$ contains precisely three non-collinear points, then an analogous result holds. Perhaps surprisingly, it turns out (cf. Example \ref{ex:four-noncollinear-points}) that this is the best result possible in this setting.

\begin{thm}
\label{thm:three-point-spectrum-SH-II-1}
Let $\fM$ be a type II$_1$ factor with faithful normal tracial state $\tau$, and let $\A$ be a MASA in $\fM$. Let $N \in \fM$ be a normal operator such that $\sigma(N)$ contains precisely three non-collinear points.  Then
\[
\overline{\D_\A(N)}^{\|\cdot\|} = \{A \in \mathcal{A} \mid \tau(A) = \tau(N), \sigma(A) \subseteq \cconv(\sigma(N))\}.
\]
\end{thm}

We also establish similar results for normal operators in von Neumann factors of type II$_\infty$ and III, and for normal operators in Cuntz C*-algebras with two generators.

The key idea behind the above results is a theorem that can be seen as an approximate multivariate analogue of Kadison's Carpenter Theorem (cf. \cites{K2002-1,K2002-2}).

\begin{thm}[Approximate Multivariate Carpenter Theorem] \label{thm:multivariate-carpenter}
Let $\fM$ be a von Neumann factor of type I$_\infty$, II or III, and let $\A$ be a MASA in $\fM$ with corresponding conditional expectation $E_\A : \fM \to \A$. Let $\{A_k\}_{k=1}^n$ be positive elements in $\A$ such that $\sum_{k=1}^n A_k = I_{\fM}$. 
Suppose that 
\begin{itemize}
\item $\fM$ is separable if $\fM$ is of type I$_\infty$, and
\item $E_\A$ is normal except possibly when $\fM$ is type I$_\infty$ and $\A$ is continuous.
\end{itemize}
Then for every $\epsilon > 0$ there are pairwise orthogonal projections $\{P_k\}_{k=1}^n$ in $\fM$ such that
\begin{enumerate}
\item $\sum_{k=1}^n P_k = I_{\fM}$,
\item  $\|A_k - E_\A(P_k)\| < \epsilon$,
\item if $\fM$ is of type I$_\infty$, then $\sigma(P_k)= \sigma_e(P_k)$, and
\item  if $\fM$ is of type II, then $\tau(A_K) = \tau(P_k)$.
\end{enumerate}
\end{thm}

The proof of Theorem \ref{thm:multivariate-carpenter} is divided into cases depending on the type of the von Neumann factor $\fM$. We also establish a similar result for UHF C*-algebras.

The intractability of the Schur-Horn problem for normal matrices can be explained (cf. Example \ref{ex:no-finite-dim-multivariate-carpenter}) by the fact that there is no finite dimensional analogue of the approximate multivariate carpenter theorem. 

We note that Massey and Ravichandran \cite{MR2014} recently and independently obtained results that are similar to the results in our paper. We also mention the recent paper \cite{KS2014}, where we establish a Schur-Horn type theorem for arbitrary operators in a II$_1$ factor by considering singular values instead of spectrum.

In addition to this introduction, there are three other sections. In Section \ref{sec:background} we briefly review a required matricial result and develop some useful terminology. In Section \ref{sec:multivariate-carpenter} we establish Theorem \ref{thm:multivariate-carpenter}, the approximate multivariate carpenter theorem. In Section \ref{sec:schur-horn} we establish our results on the Schur-Horn problem for normal operators.

\section{Background and Preliminaries} \label{sec:background}

For $n \in \bN$, let $\M_n(\bC)$ denote the $n \times n$ matrices over $\bC$, let $\D_n$ denote the diagonal subalgebra, and let $E_n : \M_n(\bC) \to \D_n$ denote the conditional expectation of $\M_n(\bC)$ onto $\D_n$. For $T \in \M_n(\bC)$, $\operatorname{tr}(T)$ will denote the (standard) trace of $T$, and $\tau(T)$ will denote the normalized trace of $T$, i.e.
\[
\tau(T) = \frac{1}{n} \operatorname{tr}(T).
\]

We require the following well-known and elementary result about matrices. For the sake of completeness, we include a short proof.

\begin{thm}
\label{thm:matrix-constant-diagonal}
Let $N \in \M_n(\bC)$ be a normal matrix.  Then there exists a unitary operator $U \in \M_n(\bC)$ such that every diagonal entry of $U^*BU$ is $\tau(B)$ for all $B \in \ca(N)$.
\end{thm}

\begin{proof}
Since $N$ is a normal matrix, there exists a unitary $W$ such that $W^*NW$ is diagonal. Fix $B \in \ca(N)$. Note $W^*BW$ is also diagonal. Let $\zeta_n = e^{\frac{2\pi i}{n}}$ and let
\[
V = \frac{1}{\sqrt{n}} \left[  \begin{array}{ccccc} (\zeta_n^1)^1 & (\zeta_n^2)^1 & \ldots & (\zeta_n^{n-1})^1 & (\zeta_n^n)^1 \\ (\zeta_n^1)^2 & (\zeta_n^2)^2 & \ldots & (\zeta_n^{n-1})^2 & (\zeta_n^n)^2 \\ \vdots & \vdots & & \vdots & \vdots \\ (\zeta_n^1)^{n-1} & (\zeta_n^2)^{n-1} & \ldots & (\zeta_n^{n-1})^{n-1} & (\zeta_n^n)^{n-1} \\ (\zeta_n^1)^n & (\zeta_n^2)^n & \ldots & (\zeta_n^{n-1})^n & (\zeta_n^n)^n  \end{array} \right].
\]
It is easy to check that $V$ is a unitary matrix and the diagonal entries of the matrix $V^*W^*BWV$ are equal to $\tau(B)$. Since $B$ was arbitrary we can take $U = WV$.
\end{proof}

To simplify discussions of Theorem \ref{thm:multivariate-carpenter}, we make the following definition.

\begin{defn}
Let $\fN$ be a von Neumann algebra.  A finite collection of elements $\{A_k\}^n_{k=1} \subseteq \fN$ is said to be a partition of unity in $\fN$ if $0 \leq A_k \leq I_\fN$ for all $k \in \{1,\ldots, n\}$ and $\sum^n_{k=1} A_k = I_\fN$.
\end{defn}

\section{Approximate Multivariate Carpenter Theorem} \label{sec:multivariate-carpenter}

In this section we will prove Theorem \ref{thm:multivariate-carpenter}. The proof will be divided into cases depending on the type of the underlying von Neumann factor. But first, we will show that no analogue of Theorem \ref{thm:multivariate-carpenter} can hold in finite dimensions.

\subsection{Type $I_n$ factors}

The following example shows that no analogue of Theorem \ref{thm:multivariate-carpenter} holds in finite dimensions, i.e. for von Neumann factors of type $I_n$.

\begin{example} \label{ex:no-finite-dim-multivariate-carpenter}
In \cite{A2007}*{Lemma 4}, it is shown that if $N = \diag(0, 1, i)$ and $A = \diag\left(\frac{1}{2}, \frac{i}{2}, \frac{1+i}{2}\right)$, then although $\sigma(A) \subseteq \cconv(\sigma(N))$ there does not exist a unitary $U \in \M_3(\bC)$ such that $E_3(U^*NU) = A$.
 
Consider the projections
\[
P_1 = \diag(1,0,0), P_2 = \diag(0,1,0), \AND P_3 = \diag(0,0,1),
\]
along with the positive contractions
\[
A_1 = \diag\left(\frac{1}{2}, \frac{1}{2}, 0 \right), A_2 = \diag\left(\frac{1}{2}, 0, \frac{1}{2}\right), \AND A_3 = \diag\left(0, \frac{1}{2}, \frac{1}{2} \right).
\]
Clearly these matrices all belong to $\D$, $\operatorname{tr}(A_k) = 1$ for $k=1,2,3$ and
\[
A_1 + A_2 + A_3 = I_3.
\]
Hence by Kadison's Carpenter Theorem there are unitaries $U_1,U_2,U_3 \in \M_3(\bC)$ such that $E_3(U_k^*P_kU_k) = A_k$ for $k=1,2,3$.

Now suppose that for any $\epsilon > 0$ there is a unitary operator $U \in \M_3(\bC)$ such that
\[
\left\|A_k - E_3(U^*P_kU)\right\| < \epsilon
\]
for $k=1,2,3$. Then, since the unitary group of $\M_3(\bC)$ is compact, there must be a unitary $U \in \M_3(\bC)$ such  that $A_k = E_3(U^*P_kU)$ for $k=1,2,3$. However, writing $N = 0P_1 + 1P_2 + iP_3$ gives 
\[
E_3(U^*NU) = \diag\left(\frac{1}{2}, \frac{i}{2}, \frac{1+i}{2}\right) = A,
\]
which is a contradiction.
\end{example}

\subsection{Type I$_\infty$ factors}

We now consider the case of the separable von Neumann factor of type I$_\infty$. This is the algebra $\B(\H)$ of bounded operators acting on a separable infinite dimensional Hilbert space $\H$.

\subsubsection{The case of the discrete MASA}

We first consider the case of the (discrete) diagonal MASA $\D$ in $\B(\H)$. We will require the following approximation result for partitions of unity in finite dimensions that are nearly scalar operators.

\begin{lem}
\label{lem:matrix-diagonal-one-different}
Let $\{\alpha_{k}\}^n_{k=1}, \{\beta_k\}^n_{k=1} \subseteq [0,1]$ be such that $\sum^n_{k=1} \alpha_k = 1 = \sum^n_{k=1} \beta_k$. Then for every $\epsilon > 0$ there exists $\ell \in \bN$ and pairwise orthogonal projections $\{P_k\}^n_{k=1} \in \M_\ell(\bC)$ such that $\sum^n_{k=1} P_k = I_\ell$ and
\[
\left\|E_\ell(P_k) - \diag(\alpha_k, \beta_k, \beta_k, \ldots, \beta_k)\right\| < \epsilon \text{ for all }k.
\]
\end{lem}
\begin{proof}
By applying a small perturbation, we may assume that $\alpha_k$ and $\beta_k$ are rational for all $k$.  Fix $\epsilon > 0$.  Choose $N_1, N_2 \in \bN$ and $\{\alpha'_k\}^n_{k=1}, \{\beta'_k\}^n_{k=1} \subseteq \bN \cup \{0\}$ such that $\alpha_k = \frac{1}{N_1} \alpha'_k$ and $\beta_k = \frac{1}{N_2} \beta'_k$ for all $k$. Choose $N_0 \in \bN$ such that
\[
\left| \frac{\alpha_k + N_0 \beta'_k}{N_0N_2 + 1} - \beta_k  \right| < \epsilon
\]
for all $k$, and let $\ell = N_1 + (N_1 - 1)N_0N_2 \in \bN$. 

Notice that
\[
\M_{N_1}(\bC) \oplus \M_{N_0N_2}(\bC)^{\oplus (N_1-1)}
\]
can be identified with a matrix subalgebra of $\M_\ell(\bC)$ such that
\[
\D_{N_1} \oplus \D_{N_0N_2}^{\oplus (N_1-1)}  = \D_\ell.
\]
 
Choose two collections of pairwise orthogonal projections 
\[
\{Q'_k\}^n_{k=1} \subseteq \D_{N_1} \qand \{Q''_{k}\}^n_{k=1} \subseteq \D_{N_0N_2}(\bC)
\]
such that 
\[
\mathrm{rank}(Q'_k) = \alpha'_k \qand \mathrm{rank}(Q''_k) = N_0\beta'_k
\]
for all $k$. Let
\[
Q_k = Q'_k \oplus (Q''_k)^{\oplus (N_1 - 1)}.
\]
It is clear that $\{Q_k\}^n_{k=1}$ are pairwise orthogonal projections summing to $I_\ell$. 
 
By Theorem \ref{thm:matrix-constant-diagonal} there exists a unitary $W_1 \in \M_{N_1}(\bC)$ such that each diagonal entry of $W_1^*Q'_kW_1$ is the average of the diagonal entries of $Q'_k$, which is 
\[
\frac{1}{N_1}(\mathrm{rank}(Q'_k)) = \alpha_k.
\]
Thus the $Q_k$ are simultaneously unitarily equivalent to projections of the form
\[
A_k \oplus (Q''_k)^{\oplus (N_1 - 1)}
\]
where each diagonal entry of $A_k \in \M_{N_1}(\bC)$ is $\alpha_k$. 

For each $j  = 1,\ldots, N_1-1$, by pairing the $(j+1)$-th index in $\M_{N_1}(\bC)$ with the $j$-th matrix entry in the direct summand $\M_{N_0N_2}(\bC)^{\oplus (N_1 - 1)}$, we further obtain that the $Q_k$ are jointly unitarily equivalent in $\M_\ell(\bC)$ to block matrices of the form
\[
\left[  \begin{array}{ccccc} \alpha_k & A_{k,1, 2} & A_{k,1,3} & \cdots & A_{k,1,N_1} \\ A_{k,2,1} & D_k & A_{k,2,3} & \cdots & A_{k,2,N_1} \\ A_{k,3,1} & A_{k,3,2} & D_k & \cdots & A_{k,3,N_1} \\ \vdots & \ddots & \ddots & \ddots & \vdots \\ A_{k,N_1, 1} & A_{k,N_1, 2} & A_{k,N_1, 3} & \cdots & D_k \end{array} \right]
\]
where $D_k \in \M_{N_0N_2+1}(\bC)$ is a diagonal matrix with $\alpha_k$ appearing once along the diagonal, $1$ appearing $\mathrm{rank}(Q''_k) = N_0\beta'_k$ times along the diagonal, and all other diagonal entries being zero.  

Applying Theorem \ref{thm:matrix-constant-diagonal} again, there exists a unitary $W_2 \in \M_{N_0N_2+1}(\bC)$ such that each diagonal entry of $W_2^*D_kW_2$ is the average of the diagonal entries of $D_k$, which is
\[
\frac{\alpha_k + N_0 \beta'_k}{N_0N_2 + 1}.
\]
It follows that conjugating the above matrices with the unitary
\[
1 \oplus W_2^{\oplus (N_1-1)} \in \bC \oplus \M_{N_0N_2+1}(\bC),
\]
results in projection matrices $P_k$ satisfying the desired conclusion.
\end{proof}

\begin{thm}
\label{thm:approx-carpenter-BH}
Theorem \ref{thm:multivariate-carpenter} holds in the case $\fM = \B(\H)$ and $\A$ is a discrete MASA.
\end{thm}

\begin{proof}
We begin with several perturbations to place the $A_k$s into a more desirable form so that Lemma \ref{lem:matrix-diagonal-one-different} may be applied.  Since the result is trivial for $n = 1$, we assume $n \geq 2$. Fix $\epsilon > 0$. Let $\{e_m\}_{m\geq1}$ be an orthonormal basis for $\H$ corresponding to $\D$. By perturbing each $A_k$ if necessary, we may assume (at a cost of $2\epsilon$ to our approximation) that $2\epsilon I_\H \leq A_n \leq (1-2\epsilon) I_\H$.

Since a diagonal self-adjoint operator has at most countably many eigenvalues of finite multiplicity, there is $N \in \bN$ such that the diagonal entries of $A_k$ with index greater than $N$ are within $\frac{\epsilon}{n}$ of the essential spectrum of $A_k$ for all $k = 1,\ldots,n-1$. Since $2\epsilon I_\H \leq A_n \leq (1-2\epsilon) I_\H$, we may then, at a cost of $\epsilon$ to our approximation, perturb the diagonal entries of each $A_k$ so that for $k =1,\ldots,n-1$ the diagonal entries of $A_k$ with index greater than $N$ belong to the essential spectrum of $A_k$ and $\epsilon I_\H \leq A_n \leq (1-\epsilon) I_\H$.

By the above paragraphs, we have reduced the problem to the case where $\epsilon I_{\H} \leq A_n \leq (1-\epsilon)I_{\H}$ and $\H$ decomposes as $\H = \H_0 \oplus \H'$, where $\H_0$ is a finite dimensional Hilbert space, such that with respect to this decomposition of $\H$ we can write
\[
A_k = D_k \oplus A'_k,
\]
where $D_k$ and $A'_k$ are diagonal operators satisfying
\[
\sum^n_{k=1} D_k = I_{\H_0}, \quad \sum^n_{k=1} A'_k = I_{\H'}, \qand \sigma(A'_k) = \sigma_e(A'_k) \text{ for }k \neq n.
\]
 
Let $\{e'_m\}_{m\geq 1}$ be the orthonormal basis of $\H'$ obtained from $\{e_m\}_{m \geq 1}$ by removing the first $N$ elements and reindexing without reordering. 

For $k = 1,\ldots,n-1$ let $\{\alpha_{k,j}\}^{\ell_k}_{j=1} \subseteq \sigma(A'_k)$ be an $\frac{\epsilon}{n}$-cover.  By perturbing both $A'_k$ and $A'_n$ simultaneously in each diagonal entry by at most $\frac{\epsilon}{n}$, we may assume, at a cost of $\epsilon$ to the final approximation, and of losing the inequality $\epsilon I_{\H'} \leq A'_n \leq (1-\epsilon)I_{\H'}$, that
\[
A'_k = \diag(\beta_{k,1}, \beta_{k,2}, \ldots)
\]
where $\beta_{k,m} \in \{\alpha_{k,j}\}^{\ell_k}_{j=1}$ for $k =1,\ldots, n-1$.
 
For each $(i_1, \ldots, i_{n-1}) \in \prod^{n-1}_{k=1} \{1,\ldots, \ell_j\}$, let
\[
I_{(i_1, \ldots, i_{n-1})} = \{j \in \bN \, \mid \, \beta_{k,j} = \alpha_{k,i_j} \mbox{ for all } k  = 1,\ldots, n-1\}.
\]
Note that $\bN$ is a disjoint union of the $I_{(i_1, \ldots, i_{n-1})}$.  Therefore, by combining elements of $\{e'_m\}_{m\geq 1}$ according to which set $I_{(i_1, \ldots, i_{n-1})}$ the index $m$ belongs to, we can write
\[
\H = \H_0 \oplus \H'_0 \oplus \left( \oplus^\ell_{j=1} \H_j\right) 
\]
where $\ell \in \bN$, $\H_0$ and $\H'_0$ are finite dimensional Hilbert spaces, each $\H_j$ is an infinite dimensional Hilbert space, and such that with respect to this decomposition of $\H$,
\[
A_k = D_k \oplus D'_k \oplus \left( \oplus^\ell_{j=1} \gamma_{k,j} I_{\H_j}  \right)
\]
where $D_k$ and $D'_k$ are each diagonal operators,
\[
\sum^n_{k=1} D_k = I_{\H_0}, \quad \sum^n_{k=1} D'_j = I_{\H'_0},
\]
and
\[\{\gamma_{k,j}\, \mid \, j =1,\ldots, \ell, k =1,\ldots, n\} \subseteq [0,1]
\]
with $\sum^n_{k=1} \gamma_{k,j} = 1$ for each $j$. Note that although our assumptions were only explicitly$A_1, \ldots A_{n-1}$, we necessarily obtain the correct form for $A_n$ because of the hypothesis that $\sum^n_{k=1} A_k = I_\H$.
 
By increasing $\ell$ if necessary, by decomposing one of the infinite dimensional Hilbert spaces occurring in the above decomposition into multiple infinite dimensional Hilbert spaces, we may assume that $\mathrm{dim}(\H_0 \oplus \H'_0) \leq \ell$.  This decomposition into direct sums further reduces the problem to one of two cases: either
\[
A_k = \beta_k I_{\mathcal{K}} \in \B(\mathcal{K}) \text{ for all } k,
\]
where $\mathcal{K}$ is a infinite dimensional Hilbert space and $\beta_k \in [0,1]$ satisfy $\sum^n_{k=1} \beta_k = 1$, or
\[
A_k = \alpha_k \oplus \beta_k I_{\mathcal{K}} \in \B(\bC \oplus \mathcal{K}) \text{ for all } k
\]
where $\mathcal{K}$ is a infinite dimensional Hilbert space and $\alpha_k, \beta_k \in [0,1]$ satisfy $\sum^n_{k=1} \beta_k = \sum^n_{k=1} \alpha_k = 1$.
 
For the first case, we may apply Lemma \ref{lem:matrix-diagonal-one-different} with $\alpha_k = \beta_k$ to obtain $m \in \bN$ and pairwise orthogonal projections $\{P_k\}^n_{k=1} \in \M_m(\bC)$ such that $\sum^n_{k=1} P_k = I_m$, and such that each entry along the diagonal of $P_k$ is within $\epsilon$ of $\beta_k$. In this case, the result now follows by taking an infinite direct sum of the $P_k$.  To ensure that $\sigma(P_k) = \sigma_e(P_k)$, we can perturb each $\beta_k$ so that $2\epsilon \leq \beta_k \leq 1-2\epsilon$ and require the approximation from Lemma \ref{lem:matrix-diagonal-one-different} to be within $\epsilon$ of the desired matrices.  This implies that the projections with the desired diagonal are non-trivial. Thus, taking their direct sum gives projections with spectrum equal to the essential spectrum.
 
For the second case, we may apply Lemma \ref{lem:matrix-diagonal-one-different} to obtain $m \in \bN$ and pairwise orthogonal projections $\{P_k\}^n_{k=1} \in \M_m(\bC)$ such that $\sum^n_{k=1} P_k = I_m$, the first diagonal entry of each $P_k$ is within $\epsilon$ of $\alpha_k$, and such that the other entries along the diagonal of $P_k$ are each within $\epsilon$ of the corresponding $\beta_k$.  The result in this case now follows from the first case by excluding the first $m$ basis vectors of $\K$.
\end{proof}

\subsubsection{The case of a continuous MASA}

For the case of a continuous MASA in $\B(\H)$, we first deal with the case when each $A_k$ is a scalar multiple of the identity.

\begin{lem}
\label{lem:continuous-masa-scalar-case}
Theorem \ref{thm:multivariate-carpenter} holds in the case $\fM = \B(\H)$, $\A$ is a continuous MASA, and $A_k = \alpha_k I_\H$ for each $k$, where the $\alpha_k \in [0,1]$ satisfy $\sum^n_{k=1} \alpha_k = 1$.

\end{lem}
\begin{proof}
Without loss of generality, we may assume that $\alpha_k = \frac{a_k}{N}$ for each $k$, where the $a_k \in \mathbb{N}\cup \{0\}$ satisfy $\sum^n_{k=1} a_k = N$.  Since $\A$ is a continuous MASA, there exist pairwise orthogonal equivalent projections $\{Q'_j\}^N_{j=1} \subseteq \A$ such that $\sum^N_{j=1} Q'_j = I_\H$.  These projections give rise to a copy of $\M_N(\bC)$ in $\B(\H)$.  Choose pairwise orthogonal projections $\{Q_k\}^n_{k=1} \subseteq \A$ such that each $Q_k$ is a sum of $a_k$ distinct $Q'_j$s.  Theorem \ref{thm:matrix-constant-diagonal} implies there exists a unitary $U \in \M_N(\bC) \subseteq \B(\H)$ such that $Q'_jU^*Q_kUQ'_j = \alpha_k Q'_j$ for all $j,k$.  Defining $P_k = U^*Q_kU$ yields the desired projections.
\end{proof}

\begin{thm}
\label{thm:continuous-masa-approx-carpenter}
Theorem \ref{thm:multivariate-carpenter} holds in the case $\fM = \B(\H)$ and $\A$ is a continuous MASA.

\end{thm}
\begin{proof}
By approximation, we may assume that the spectrum of each $A_k$ is finite.  By taking compressions of all possible intersections of the spectral projections of the $A_k$'s (which are necessarily infinite projections since $\A$ is continuous), we can apply Lemma \ref{lem:continuous-masa-scalar-case} to each compression. Taking a sum of these finite number of compressions then yields the result.
\end{proof}

\subsection{Type II$_1$ factors}

For the case of a von Neumann factor of type II$_1$, we first consider the case when each $A_k$ is a positive scalar. For this, we require the following two technical results about approximating real numbers.

\begin{lem} \label{lem:approx-reals-1}
Given $\epsilon > 0$, $n \in \bN$, and $\{r_k\}^n_{k=1} \subseteq (0,1)$, there exists $\{q_k\}^n_{k=1} \subseteq \mathbb{Q}$ such that
\begin{enumerate}
\item $q_k \in \left[ \frac{1}{2} r_k, r_k\right)$, 
\item $\left|\frac{r_k}{\sum^n_{j=1} r_j} - \frac{q_k}{\sum^n_{j=1} q_j}\right| < \epsilon$, and
\item $\left|\frac{r_k}{\sum^n_{j=1} r_j} - \frac{r_k - q_k}{\sum^n_{j=1} r_j - q_j}\right| < \epsilon$ for all $k$
\end{enumerate}
\end{lem}
\begin{proof}
Choose $0 < \delta \leq \min\left\{\left. \frac{1}{2} r_k \, \right| \, k = 1,\ldots, n\right\}$ 
and for each $k$ choose $q_k \in \left[\frac{1}{2} r_k, \frac{1}{2} r_k + \delta\right) \cap \mathbb{Q}$.  Then
\[
\frac{r_k}{2n\delta + \sum^n_{j=1} r_j} =\frac{\frac{1}{2} r_k}{\sum^n_{j=1} \left(\frac{1}{2}r_k + \delta\right)} \leq \frac{q_k}{\sum^n_{j=1} q_j} \leq \frac{\frac{1}{2} r_k + \delta}{\sum^n_{j=1} \frac{1}{2}r_j} = \frac{r_k + 2\delta}{\sum^n_{j=1} r_j}
\]
and
\begin{align*}
\frac{r_k - 2\delta}{\sum^n_{j=1} r_j} = \frac{r_k - \frac{1}{2}r_k - \delta}{\sum^n_{j=1} r_j - \frac{1}{2}r_j} &\leq \frac{r_k - q_k}{\sum^n_{j=1} r_j - q_j} \\
& \leq \frac{r_k - \frac{1}{2}r_k}{\sum^n_{j=1} \left(r_j - \left(\frac{1}{2}r_j + \delta\right)\right)} \\
& = \frac{r_k}{-2n\delta + \sum^n_{j=1} r_j} 
\end{align*}
for each $k$.  Therefore, since $\{r_k\}_{k\geq 1}$ are fixed, the result follows by choosing $\delta$ sufficiently small.
\end{proof}
\begin{lem}
\label{lem:approx-reals-2}
Given $\epsilon > 0$, $n \in \bN$, and $(p_k)^n_{k=1} \subseteq (0,1)$ such that $\sum^n_{k=1} p_k = 1$, there exists $\{q^{(k)}_j \mid j \in \bN, k =1,\ldots, n \} \subseteq [0,1] \cap \mathbb{Q}$ such that
\[
\sum_{j\geq 1} q^{(k)}_j = p_k \qand
\left| p_k - \frac{q^{(k)}_j}{\sum^n_{k=1} q^{(k)}_j}  \right| < \epsilon
\]
for all $j \in \bN$ and for all $k$.
\end{lem}
\begin{proof}
Fix $\epsilon > 0$.  Applying Lemma \ref{lem:approx-reals-1} with $r_k = p_k$ for all $k$, we obtain $q^{(1)}_1, \ldots, q^{(n)}_1 \in \mathbb{Q}$ such that $q^{(k)}_1 \in \left[\frac{1}{2} p_k, p_k\right)$,
\[
\left|p_k - \frac{q^{(k)}_1}{\sum^n_{j=1} q^{(j)}_1}\right| < \frac{1}{2}\epsilon, \qand
\left|p_k - \frac{p_k - q^{(k)}_1}{\sum^n_{j=1} p_j - q^{(j)}_1}\right| < \frac{1}{2}\epsilon.
\]
 
Next, we may apply Lemma \ref{lem:approx-reals-1} with $r_k = p_k - q^{(k)}_1$ for all $k$ to obtain $q^{(1)}_2, \ldots, q^{(n)}_2 \in \mathbb{Q}$ such that 
\[
q^{(k)}_2 \in \left[\frac{1}{2} \left(p_k-q^{(k)}_1\right), p_k-q^{(k)}_1\right)
\]
(so that $q^{(k)}_1 + q^{(k)}_2 \in \left[\frac{3}{4} p_k, p_k\right)$),
\begin{align*}
\left|\frac{p_k - q^{(k)}_1}{\sum^n_{j=1} p_j - q^{(j)}_1} - \frac{q^{(k)}_2}{\sum^n_{j=1} q^{(j)}_2}\right| &< \frac{1}{4}\epsilon, \AND \\
\left|\frac{p_k - q^{(k)}_1}{\sum^n_{j=1} p_j - q^{(j)}_1} - \frac{p_k - q^{(k)}_1- q^{(k)}_2}{\sum^n_{j=1} p_j - q^{(j)}_1- q^{(j)}_2}\right| &< \frac{1}{2}\epsilon.
\end{align*}
Thus
\[
\left|p_k - \frac{q^{(k)}_2}{\sum^n_{j=1} q^{(j)}_2}\right| < \left(\frac{1}{2} + \frac{1}{4}\right)\epsilon
\]
for all $k$.  

By recursively applying Lemma \ref{lem:approx-reals-1} with $r_k = p_k - \sum^\ell_{j=1} q^{(k)}_j$ at the $\ell^{\mathrm{th}}$-step, we will obtain $\{q^{(k)}_j \mid j \in \bN, k =1,\ldots, n \} \subseteq [0,1] \cap \mathbb{Q}$ such that
\[
\sum^\ell_{j=1}q^{(k)}_j \in \left[\sum^\ell_{j=1} \frac{1}{2^j} p_k, p_k\right)\AND
\left|p_k - \frac{q^{(k)}_\ell}{\sum^n_{j=1} q^{(j)}_\ell}\right| < \left(\sum^\ell_{j=1} \frac{1}{2^j}\right)\epsilon
\]
for all $k$ and $\ell \in \bN$.  Hence the result follows.
\end{proof}

\begin{lem}
\label{lem:scalar-case-II-1}
Theorem \ref{thm:multivariate-carpenter} holds in the case that $\fM$ is a II$_1$ factor and $A_k = \alpha_k I_\fM$ for some $\{\alpha_k\}^n_{k=1} \subseteq [0,1]$ with $\sum^n_{k=1} \alpha_k = 1$.
\end{lem}
\begin{proof}
Without loss of generality, $\alpha_k \neq 0$ for all $k$.  
By Lemma \ref{lem:approx-reals-2} there exists  $\{q^{(k)}_j \mid j \in \bN, k =1,\ldots, n \} \subseteq [0,1] \cap \mathbb{Q}$ such that
\[
\sum_{j\geq 1} q^{(k)}_j = \alpha_k
\]
for all $k$ and, if
\[
\gamma_{j,k}  = \frac{q^{(k)}_j}{\sum^n_{k=1} q^{(k)}_j}
\]
then $\left| \alpha_k - \gamma_{j,k}  \right| < \epsilon$ for all $j \in \bN$ and $k \in \{1,\ldots, n\}$. 

Since $\fM$ is a type II$_1$ factor, there exists pairwise orthogonal projections 
\[
\left\{ \left. Q^{(k)}_j \right| j \in \bN, k =1,\ldots, n \right\} \subseteq \A
\]
such that $\tau\left(Q^{(k)}_j\right) = q^{(k)}_j$.  For each $k$ let
\[
P_k = \sum_{j\geq 1} Q^{(k)}_j
\]
(the sum in the strong operator topology) and for each $j \in \bN$ let
\[
Q_j = \sum^n_{k=1} Q^{(k)}_j.
\]
Note the $P_k$s and $Q_j$s are clearly projections in $\fM$ such that $\{P_k\}^n_{k=1}$ is a collection of pairwise orthogonal projections with $\sum^n_{k=1} P_k = I_\fM$ and $\tau(P_k) = \alpha_k$.  Furthermore, notice that $Q_{j_1}Q_{j_2} = 0$ if $j_1 \neq j_2$, $\sum_{j\geq 1} Q_j = I_\fM$,
and
\[
Q_jP_k = P_k Q_j = Q_jP_kQ_j = Q^{(k)}_j
\]
for each $j,k$.
 
For a fixed $j \in \bN$, consider the II$_1$ factor $Q_j \fM Q_j$.  Notice there exists $a_{j,1}, a_{j,2}, \ldots, a_{j,n}, b_j \in \bN \cup \{0\}$ with $b_j \neq 0$  such that
\[
\sum^n_{k=1} a_{j,k} = b_j \qand \tau_{Q_j \fM Q_j} \left(Q^{(k)}_j\right) = \gamma_{j,k} = \frac{a_{j,k}}{b_j}
\]
for all $k$. Since $Q_j = \sum^n_{k=1} Q^{(k)}_j$ with $Q^{(k)}_j \in \A$, by dividing $Q^{(k)}_j$ into $a_{j,k}$ mutually orthogonal projections each of trace $\frac{1}{b_j}$ and by using the property of type II$_1$ factors that projections of equal trace are Murray-von Neumann equivalent, we can construct a unital inclusion of $\fM_{b_j}(\bC)$ inside $Q_j \fM Q_j$ such that $Q_jP_kQ_j$ corresponds to a diagonal projection matrix with $1$  appearing on the diagonal precisely $a_{j,k}$ times and each diagonal projection lies in $\A$.  To simplify notation, let $R_{j,1}, \ldots, R_{j,b_j} $ denote the pairwise orthogonal projections in $Q_j\A Q_j$ obtained from the diagonal matrix units of $\M_{b_j}(\bC)$.  Therefore, by Theorem \ref{thm:matrix-constant-diagonal} there exists a unitary $U_j \in Q_j \fM Q_j$ such that
\[
R_{j,\ell}U_j^* Q_jP_kQ_jU_jR_{j,\ell} = \gamma_{j,k}  R_{j,\ell}
\]
for all $\ell, k$.
 
Let 
\[
U = \sum_{j\geq 1} U_j \in \fM
\]
(where the sum is in the strong operator topology).  It is clear that $U$ exists and is a unitary operator in $\fM$.  Furthermore, for all $j, \ell, k$, 
\[
\left\|\alpha_k R_{j,\ell} - R_{j,\ell} U^*P_kUR_{j,\ell}\right\| = |\alpha_k -\gamma_{j,k}| < \epsilon.
\]
 
Notice that $\{ R_{j,\ell} \mid j \in \bN, \ell = 1, \ldots, b_j\}$ are pairwise orthogonal projections in $\A$ such that 
\[
\sum_{j\geq 1}\sum^{b_j}_{\ell=1} R_{j,\ell} = I_\fM.
\]
Hence 
\[
\left\|\alpha_k I_\fM - E_\A(U^*P_kU)\right\| = \sup_{j \in \bN, \ell \in \{1,\ldots, b_j\}}\left\|\tau(P_k) R_{j,\ell} - R_{j,\ell} U^*P_kUR_{j,\ell}\right\| < \epsilon.
\]
Thus the proof is complete as $\{U^*P_kU\}^n_{k=1}$ is an appropriate set of projections.
\end{proof}

\begin{thm}
Theorem \ref{thm:multivariate-carpenter} holds in the case $\fM$ is a II$_1$ factor.
\end{thm}

\begin{proof}
Using the fact that $A_1,\ldots, A_n$ are contained within a diffuse, separably generated, abelian subalgebra $\A_0$ of $\A$, that $\A_0\simeq L_\infty[0,1]$, and by averaging functions over a cover of $[0,1]$ on which the essential diameter of functions does not vary substantially, there exists an $\ell \in \bN$, $\{\alpha_{k,j} \, \mid \, k \in \{1,\ldots, n\}, j \in \{1,\ldots, \ell\}\} \subseteq [0,1]$, and pairwise orthogonal projections $\{Q_j\}^\ell_{j=1} \subseteq \A_0 \subseteq \A$ such that $\sum^\ell_{j=1} Q_j = I_\fM$,
\[
\sum^n_{k=1} \alpha_{k,j} = 1, \quad
\sum^\ell_{j=1} \alpha_{k,j} \tau(Q_j) = \tau(A_k), \qand
\left\|A_k - \sum^\ell_{j=1} \alpha_{k,j} Q_j\right\| < \epsilon
\]
for all $k$.

Fix $j \in \{1, \ldots, \ell\}$.  By applying Lemma \ref{lem:scalar-case-II-1} to the type II$_1$ factor $Q_j\fM Q_j$ togehter with the MASA $Q_j\A Q_j$, there exists pairwise orthogonal projections $\left\{P^{(j)}_k\right\}^n_{k=1} \subseteq Q_j \fM Q_j$ such that $\sum^n_{k=1} P^{(j)}_k = Q_j$, $\tau\left(P^{(j)}_k \right) = \alpha_{k,j} \tau(Q_j)$, and
\[
\left\|\alpha_{k,j} Q_j - E_{Q_j\A Q_j}\left(P^{(j)}_k\right) \right\| < \epsilon \text{ for all }k
\]

For each $k$ let
\[
P_k = \sum^\ell_{j=1} P^{(j)}_k.
\]
By construction, it is clear that $\{P_k\}^n_{k=1}$ is a collection of pairwise orthogonal projections such that $\sum^n_{k=1} P_k = \sum^\ell_{j=1} Q_j = I_\fM$ and
\[
\tau(P_k) = \sum^\ell_{j=1} \alpha_{k,j} \tau(Q_j) = \tau(A_k).
\]
Finally
\[
\begin{array}{l}
\left\|E_\A(P_k)- \sum^\ell_{j=1} \alpha_{k,j} Q_j\right\| \\ 
=  \left\|\sum^\ell_{j=1} E_{Q_j\A Q_j}(Q_jP_kQ_j)- \sum^\ell_{j=1} \alpha_{k,j} Q_j\right\|\\
 = \max_{1\leq j \leq \ell}\left\{\left\|E_{Q_j\A Q_j}\left(P^{(j)}_k\right)- \alpha_{k,j} Q_j\right\| \right\} < \epsilon\\
 \end{array} 
\]
for all $k$.  Thus 
\[
\left\|E_\A(P_k) - A_k\right\| < 2\epsilon
\]
for all $k$ as desired.
\end{proof}

For certain operators in II$_1$ factors, equality can be obtained in Theorem \ref{thm:multivariate-carpenter}. This is seen in the next result, which can be obtained from a careful analysis of the above proof.

\begin{cor}
Let $\fM$ be a type II$_1$ factor and let $\A$ be a MASA in $\fM$ with corresponding normal conditional expectation $E_\A : \fM \to \A$.  Let $\{Q_j\}^m_{j=1}$ be pairwise orthogonal projections in $\A$ such that $\sum^m_{j=1} Q_j = I_\fM$ and let $\{\{\alpha_{j,k}\}^m_{j=1}\}^n_{k=1} \subseteq [0,1] \cap \mathbb{Q}$ satisfy
\[
\sum^n_{k=1} \alpha_{j,k} = 1
\]
for all $j$. Then there exists pairwise orthogonal projections $\{P_k\}^n_{k=1} \subseteq \fM$ such that $\sum^n_{k=1} P_k = I_\fM$ and
\[
E_\A(P_k) = \sum^m_{j=1} \alpha_{j,k} Q_j
\]
for all $k$.
\end{cor}
\begin{proof}
By compressing $\fM$ and $\A$ by each $Q_j$ separately, we may assume that $m = 1$.  A careful examination of the proof of Lemma \ref{lem:scalar-case-II-1} reveals that if the trace of each projection in the hypothesis is rational, then equality may be obtained in the conclusion. Hence the condition that each $\alpha_{j,k}$ is rational implies the result.
\end{proof}

To further examine when Theorem \ref{thm:multivariate-carpenter} holds with equality, we require the following three definitions.

\begin{defn}
Let $\fM$ be a type II$_1$ factor and let $\A$ be a MASA in $\fM$.  The normalizer of $\A$ in $\fM$, denoted by $\mathcal{N}_\A$, is the subgroup of the unitary group $\U(\fM)$ of $\fM$ defined by
\[
\mathcal{N}_\A := \{U \in \U(\fM) \, \mid \, U^* \A U = \A\}.
\]
The MASA $\A$ is said to be semiregular if $(\mathcal{N}_\A)'' \cap \fM$ is a factor.  If, in addition, $(\mathcal{N}_\A)'' \cap \fM = \fM$, then $\A$ is said to be regular (or Cartan).
\end{defn}

\begin{defn}[\cite{P1983}]
Let $\fM$ be a II$_1$ factor with trace $\tau$ and let $\fN_1$ and $\fN_2$ be von Neumann algebras of $\fM$.  If $E_{\fN_k} : \fM \to \fN_k$ are the normal conditional expectations, it is said that $\fN_1$ and $\fN_2$ are orthogonal if any of the following equivalent conditions hold:
\begin{enumerate}
\item $\tau(T_1T_2) = 0$ for all $T_1 \in \fN_1$, $T_2 \in \fN_2$ with $\tau(T_k) = 0$.
\item $\tau(T_1T_2) = \tau(T_1)\tau(T_2)$ for all $T_1 \in \fN_1$ and $T_2 \in \fN_2$.
\item $E_{\fN_1}(T_2) = \tau(T_2)$ for all $T_2 \in \fN_2$.
\end{enumerate}
\end{defn}

\begin{defn}
Let $\fM$ be a type II$_1$ factor.  A MASA $\A \subseteq \fM$ is said to be totally complementable if for every projection $P \in \A$ the MASA $P\A P$ of $P\fM P$ admits a diffuse orthogonal abelian subalgebra.
\end{defn}
\begin{rem}
It is not difficult to see that the Cartan MASA in the hyperfinite II$_1$ factor is totally complementable and thus, by \cite{P1981}*{Proposition 3.6}, so is every semiregular MASA in a type II$_1$ factor with separable predual.
\end{rem}
With the above definitions in hand, we are now able to establish a multivariate generalization of \cite{AM2009}*{Theorem 3.2}. We note that the proof is virtually identical.
\begin{thm}
Let $\fM$ be a type II$_1$ factor and let $\A$ be a totally complementable MASA in $\fM$ with corresponding normal conditional expectation $E_\A : \fM \to \A$.  Let $\{A_k\}^n_{k=1}$ be a partition of unity in $\A$ such that each $\sigma(A_k)$ contains a finite number of points.
Then there exists a collection of pairwise orthogonal projections $\{P_k\}^n_{k=1} \subseteq \fM$ such that $\sum^n_{k=1} P_k = I_\fM$, and
and $E_\A(P_k)=A_k$ for all $k$.
\end{thm}
\begin{proof}
By the assumptions on the $A_k$'s, there exists a set of pairwise orthogonal, non-zero projections $\{Q_j\}^m_{j=1} \subseteq \fM$ summing to the identity and scalars 
\[
\{\alpha_{k,j} \, \mid \, k =1,\ldots, n, j = 1,\ldots, m\}\subseteq [0,1]
\]
satisfying $A_k = \sum^m_{j=1} \alpha_{k,j}Q_j$ for all $k$ and $\sum^n_{k=1} \alpha_{k,j} = 1$ for all $j$.
 
Since $\A$ is totally complementable in $\fM$, for each $j$ there exists a diffuse abelian von Neumann subalgebra $\B_j$ of the type II$_1$ factor $Q_j \fM Q_j$ such that $Q_j\A Q_j$ and $\B_j$ are orthogonal.  Since $\B_j$ is diffuse, there exists pariwise orthogonal projections $\{P_{k,j}\}^n_{k=1}$ summing to $Q_j$ such that $\tau_{Q_j\fM Q_j}(P_{k,j}) = \alpha_{k,j}$ for all $k$.  Therefore, since $Q_j\A Q_j$ and $\B_j$ are orthogonal, 
\[
E_{Q_j\A Q_j}(P_{k,j}) =\tau_{Q_j\fM Q_j}(P_{k,j})Q_j = \alpha_{k,j}Q_j.
\]
 
For each $k$, let $P_k = \sum^m_{j=1} P_{k,j}$.  Therefore $\{P_k\}^n_{k=1}$ is a collection of pariwise orthogonal projections summing to the identity such that
\[
E_\A(P_k) = \sum^m_{j=1} E_{Q_j\A Q_j}(P_{k,j}) = \sum^m_{j=1} \alpha_{k,j}Q_j = A_k. \qedhere
\]
\end{proof}

Of course, approximate results yield precise results in the setting of an ultraproduct.

\begin{thm}
Let $\fM$ be a type II$_1$ factor, let $\A$ be a MASA is $\fM$, let $\omega$ be a free ultrafilter on $\bN$, and let $\fM^\omega$ and $\A^\omega$ be the ultraproducts of $\fM$ and $\A$ respectively.  If $\{A_k\}^n_{k=1}$ is a partition of unity in $\A^\omega$ and $E_{\A^\omega} : \fM^\omega \to \A^\omega$ is the normal conditional expectation of $\fM^\omega$ onto $\A^\omega$, then there exists pairwise orthogonal projections $\{P_k\}^n_{k=1} \subseteq \fM^\omega$ such that 
\[
E_{\A^\omega}(P_k) = A_k
\]
for all $k$.
\end{thm}
\begin{proof}
First choose a representative $(A_{1,j})_{j\geq 1} \in \A^\omega$ of $A_1$ such that $0 \leq A_{1,j} \leq I_\fM$ for all $j$.  Next, choose a representative $(A'_{2,j})_{j\geq 1} \in \A^\omega$ of $A_2$ such that $0 \leq A'_{2,j} \leq I_\fM$ for all $j$.  We desire to modify the representative $(A'_{2,j})_{j\geq 1}$ so that $0 \leq A_{1,j} + A'_{2,j} \leq I_\fM$ for all $j$.  Consider the function $f \in C([0,\infty))$ defined by
\[
f(x) = \left\{
\begin{array}{ll}
x & \mbox{if } 0 \leq x \leq 1  \\
1 & \mbox{if } x \geq 1
\end{array} \right. .
\]
For each $j$ define
\[
A_{2,j} = f(A_{1,j} + A'_{2,j}) - A_{1,j} \in \fM.
\]
Hence it is clear that $0 \leq A_{2,j} \leq I_\fM$ for all $j$, 
\[
A_{1,j} + A_{2,j} = f(A_{1,j} + A'_{2,j}) \leq I_\fM
\]
for all $j$, and
\begin{align*}
(A_{2,j})_{j\geq 1} &= (f(A_{1,j} + A'_{2,j}))_{j\geq 1} - (A_{1,j})_{j\geq 1} \\
&= f(A_1 + A_2) - A_1 = A_1 + A_2 - A_1 = A_2.
\end{align*}

For $k = 2,\ldots, n-1$, by recursively selecting a representative $(A'_{k,j})_{j\geq 1} \in \A^\omega$ of $A_k$ such that $0 \leq A'_{k,j} \leq I_\fM$ for all $j \in \bN$ and defining
\[
A_{k,j} = f(A_{1,j} + A_{2,j} + \cdots + A_{k-1, j} + A'_{k,j}) - A_{1,j} - A_{2,j} - \cdots - A_{k-1, j} \in \fM
\]
for all $j$, we obtain representatives $(A_{k,j})_{j\geq 1} \in \A^\omega$ of $A_k$ such that $0 \leq A_{k,j} \leq I_\fM$ and
\[
0 \leq \sum^{n-1}_{k=1} A_{k,j} \leq I_\fM
\]
for all $j$.  Therefore, if $A_{n,j} = I_\fM - \sum^{n-1}_{k=1} A_{k,j}$ for all $j \in \bN$, then $0 \leq A_{n,j} \leq I_\fM$ for all $j$, $\sum^n_{k=1} A_{k,j} = I_\fM$ for all $j$, and
\[
(A_{n,j})_{j\geq 1} = I_{\fM^\omega} - \sum^{n-1}_{k=1} A_k = A_n.
\]
 
Let $E_\A : \fM \to \A$ be the normal conditional expectation of $\fM$ onto $\A$.  By Theorem \ref{thm:multivariate-carpenter} and by construction, for each $j \in \bN$ there exists pariwise orthogonal projections $\{P_{k,j}\}^n_{k=1} \subseteq \fM$ that sum to the identity such that 
\[
\left\|E_\A(P_{k,j}) - A_{k,j}\right\| < \frac{1}{j}
\]
for all $k$.  Therefore, if
\[
P_k = (P_{k,j})_{j\geq 1} \in \fM^\omega,
\]
then $\{P_k\}^n_{k=1}$ are pairwise orthogonal projections that sum to $I_{\fM^\omega}$ such that
\[
E_{\A^\omega}(P_k) = (E_{\A}(P_{k,j}))_{j\geq 1} = (A_{k,j})_{j\geq 1} = A_k. \qedhere
\]
\end{proof}

\subsection{Type II$_\infty$ factors}

The proof of the II$_\infty$ case of Theorem \ref{thm:multivariate-carpenter} follows easily using previously illustrated techniques and arguments.

\begin{thm}
\label{thm:approx-carpenter-II-infty}
Theorem \ref{thm:multivariate-carpenter}  holds in the case that $\fM$ is a II$_\infty$ factor.
\end{thm}
\begin{proof}
The proof of the case where each $A_k$ is a scalar follows by using arguments from Lemma \ref{lem:continuous-masa-scalar-case} along with the fact that one can find $n$ pairwise orthogonal, infinite projections in $\A$.  The result then follows from using finite approximations of the $A_k$s, taking all possible intersections of spectral projections, using the scalar case to deal with the infinite common spectral projections, and using the II$_1$ case of Theorem \ref{thm:multivariate-carpenter} to deal with the finite common spectral projections.
\end{proof}

\subsection{Type III factors}

The III case of Theorem \ref{thm:multivariate-carpenter} is a trivial application of Theorem \ref{thm:matrix-constant-diagonal}.

\begin{thm}
\label{thm:approx-carpenter-III}

Theorem \ref{thm:multivariate-carpenter} holds in the case $\fM$ is a type III factor.

\end{thm}
\begin{proof}
The proof reduces to the case each $A_k$ is a scalar by using finite approximations.  The result then trivially follows by using arguments in Lemma \ref{lem:continuous-masa-scalar-case} along with the fact that the identity in a type III factor can be divided into an arbitrary number of pairwise orthogonal projections contained in $\A$.
\end{proof}

\subsection{UHF C*-algebras}

In this section we briefly examine how our results generalize to other classes of C$^*$-algebras.  One excellent candidate is the class of UHF C$^*$-algebras. This is due to their matrix algebra substructures, and the clear choice of a MASA -- namely, the diagonal subalgebra. 

It is tempting to speculate that an approximate multivariate carpenter theorem does not hold for any class of UHF C$^*$-algebras, since no such result holds for matrix algebras (cf. Example \ref{ex:no-finite-dim-multivariate-carpenter}). However, it turns out that this is not the case.

\begin{thm}
\label{thm:UHF-MVCT}
Let $\fA$ be an infinite UHF C$^*$-algebra (i.e. not a matrix algebra) and let $\mathfrak{D}$ be the diagonal subalgebra of $\fA$ with corresponding conditional expectation $E : \fA \to \mathfrak{D}$.  Let $\{A_k\}^n_{k=1}$ be a partition of unity in $\D$.  Then for each $\epsilon > 0$ there exists a collection of (non-zero) pairwise orthogonal projections $\{P_k\}^n_{k=1} \subseteq \fA$ such that $\sum^n_{k=1} P_k = I_\fA$ and
\[
\left\| E(P_k) - A_k\right\| <\epsilon
\]
for all $k$.
\end{thm}
\begin{proof}
By approximating, we may assume that $A_1,\ldots, A_n \in \mathfrak{D} \cap \M_m(\bC)$ where $\M_m(\bC)$ is part of the finite dimensional part in the construction of $\fA$.  Thus, by applying the following argument $m$ times, we may assume that $m = 1$ and  each $A_k$ is a (non-zero) scalar.

Consider the matrix algebra sequence
\[
\bC \hookrightarrow \M_{n_1}(\bC) \hookrightarrow \M_{n_2}(\bC) \hookrightarrow \cdots.
\]
It is clear that we can approximate $\{A_k\}^n_{k=1}$ by rational numbers whose denominators are $n_j$ for some large enough $j$.  Hence, as each $A_k$ is a diagonal scalar matrix in $\M_{n_j}(\bC)$, by Theorem \ref{thm:matrix-constant-diagonal} we can choose (non-zero) pairwise orthogonal projections with the correct diagonals to obtain the result.
\end{proof}
\begin{rem}
\label{rem:UHF-no-SH}
Note that the difference between this result and Theorem \ref{thm:multivariate-carpenter} for II$_1$ factors is the additional conclusion obtained for II$_1$ factors that $\tau(P_k) = \tau(A_k)$ for each $k$.  This cannot be accomplished in the UHF setting, since projections in matrix algebras necessarily have rational traces when viewed as elements in a UHF C$^*$-algebra.
\end{rem}

\section{Schur-Horn Type Theorems for Normal Operators} \label{sec:schur-horn}

In this section we will establish our approximate Schur-Horn type theorems for normal operators in a von Neumann factor by applying Theorem \ref{thm:multivariate-carpenter}.  Recall that given a von Neumann algebra $\fM$, a MASA $\A$ with corresponding conditional expectation $E_\A : \fM \to \A$, and a normal operator $N \in \fM$, our goal is to characterize the set of approximate diagonals
\[
\overline{\D_\A(N)}^{\|\cdot\|} := \overline{\{ E_\A(U^*NU) \mid U \text{ a unitary in }\fM \}}^{\left\|\cdot\right\|}.
\]

In order to prove a necessary condition for an operator $A$ to belong to $\overline{\D_\A(N)}^{\|\cdot\|}$, we first require the following lemma.

\begin{lem}
\label{lem:spectrum-of-expectation-in-convex-hull}
Let $\fA$ and $\fB$ be unital C$^*$-algebras with $\fA$ abelian and let $\Phi : \fB \to \fA$ be a unital positive map.  If $N \in \fB$ is a normal operator and
\[
T \in \overline{ \{ \Phi(U^*NU) \, \mid \, U \text{ a unitary in } \fB\} }^{\|\cdot\|},
\]
then $\sigma(T) \subseteq \cconv(\sigma(N))$.
\end{lem}
\begin{proof}
Suppose $T = \Phi(U^*NU)$ for some unitary $U \in \fB$.  Since $\fA$ is abelian,
\[
\sigma(T) = \{\phi(T) \, \mid \, \varphi \in \M(\fA)\},
\]
where $\M(\fA)$ denotes the set of multiplicative linear functionals.  Notice for each $\varphi \in \M(\fA)$ that $\varphi \circ \Phi$ is a state on $\fB$.  Hence, as $U^*NU$ is normal, 
\[
\varphi(T) = \varphi(\Phi(U^*NU)) \in \cconv(\sigma(U^*NU)) = \cconv(\sigma(N))
\]
for all $\varphi \in \M(\fA)$.  Hence $\sigma(T) \subseteq \cconv(\sigma(N))$.  The result then follows by the semicontinuity of the spectrum.
\end{proof}

\begin{cor}\label{cor:approx-diagonals-have-spectrum-in-convex-hull}
Let $N$ be a normal operator in a von Neumann algebra $\A$ and let $\A$ be a MASA in $\fM$ with corresponding conditional expectation $E_\A : \fM \to \A$. Then for every $A \in \overline{\D_\A(N)}^{\|\cdot\|}$,
\[
\sigma(A) \subseteq \cconv(\sigma(N)).
\]
\end{cor}

As in Section \ref{sec:multivariate-carpenter}, we will consider von Neumann factors according to their type.

\subsection{Type I$_\infty$ factors}

For the case of $\B(\H)$, we divide the discussion into two parts, depending on the type of MASA.

\subsubsection{The case of the discrete MASA}

In this section we will generalize \cite{N1999}*{Theorem 3.13} as much as seems possible.  In particular, we will establish Theorem \ref{thm:SH-BH-Discrete-Continuous} for the case of a discrete MASA.  The condition $\sigma(N) \subseteq \cconv (\sigma_e(N))$ remains necessary for us due to the lack of a Schur-Horn type theorem for normal matrices. 

We begin with the following simplification.

\begin{lem}
\label{lem:approx-SH-BH-essential}
Let $N$ be a normal operator in $\B(\H)$ such that $\sigma(N) =\sigma_e(N)$.  Then
\[
\overline{\D_\D(T)}^{\left\|\cdot\right\|} = \{T \in \D \, \mid \, \sigma(T)  \subseteq \cconv(\sigma(N))\}.
\]
\end{lem}
\begin{proof}
One inclusion follows from Lemma \ref{lem:spectrum-of-expectation-in-convex-hull} and the fact that $E_\D$ is completely positive.
 
For the other inclusion, fix $T \in \D$ such that $\sigma(T) \subseteq \cconv(\sigma(N))$. Let $\epsilon > 0$.  Since $T$ and $N$ are normal operators, it is clear that we can find normal operators $T_0 \in \D$ and $N_0 \in \B(\H)$ such that 
\begin{itemize}
\item $\sigma(T_0)$ and $\sigma(N_0)$ consist of a finite number of points,
\item $\sigma_e(N_0) = \sigma(N_0)$,
\item $\left\|T - T_0\right\| \leq \epsilon$,
\item $\left\|N - N_0\right\| \leq \epsilon$, and
\item $\sigma(T_0) \subseteq \cconv(\sigma(N_0))$.
\end{itemize}   
It suffices to prove the result for the pair $N_0$ and $T_0$.
 
Let $\sigma(T_0) = \{\alpha_j\}^m_{j=1}$ and $\sigma(N_0) = \sigma_e(N_0) = \{\beta_k\}^n_{k=1}$.  Since 
\[
\{\alpha_j\}^m_{j=1} = \sigma(T_0) \subseteq \cconv(\sigma(N_0)) = \cconv(\{\beta_k\}^n_{k=1}),
\]
and since $T_0$ is a diagonal operator, there exists 
\[
\{\gamma_{j,k}\, \mid \, j \in \bN, k = 1,\ldots, n\} \subseteq [0,1]
\]
such that $\sum^n_{k=1} \gamma_{j,k} = 1$ for all $j \in \bN$ and
\[
T_0 = \sum^n_{k=1} \beta_k \diag(\gamma_{1,k}, \gamma_{2,k}, \gamma_{3,k}, \ldots).
\]
Note that by perturbing each $\gamma_{1,k}$ if necessary, we may assume that $\gamma_{1,k} \in (0,1)$ for all $k$, $\sum^n_{k=1} \gamma_{1,k} = 1$, and
\[
\left\|T_0 - \sum^n_{k=1} \beta_k \diag(\gamma_{1,k}, \gamma_{2,k}, \gamma_{3,k}, \ldots)\right\| < \epsilon.
\]
 
Let 
\[
M = \sup_k |\beta_k| + 1 \qand
\gamma = \inf_k  \gamma_{1,k}  > 0.
\]
For each $k$ let 
\[
A_k = \diag(\gamma_{1,k}, \gamma_{2,k}, \gamma_{3,k}, \ldots) \in\D.
\]
Thus $\{A_k\}^n_{k=1}$ is a partition of unity in $\D$.  By Theorem \ref{thm:approx-carpenter-BH} there exists pairwise orthogonal projections $\{P_k\}^n_{k=1}$ with $\sigma(P_k) = \sigma_e(P_k)$ such that $\sum^n_{k=1} P_k = I_\H$ and
\[
\left\|A_k - E_\D(P_k)\right\| < \min\left\{\frac{\epsilon}{M}, \frac{\gamma}{2}\right\}
\]
for all $k$. Since $\left\|A_k\right\| \geq \gamma$, we see that $E_\D(P_k) \neq 0$ for each $k$ so $P_k \neq 0$ and $P_k \neq I_\H$ for all $k$.
 
Notice that
\[
\left\|T_0 - E_\D\left(\sum^n_{k=1} \beta_k P_k\right)\right\| \leq \epsilon + \sum^n_{k=1} |\beta_k| \left\|A_k - E_\D(P_k)\right\| < 2\epsilon.
\]
However, it is clear that $N' = \sum^n_{k=1} \beta_k P_k$ is a normal operator such that $\sigma(N') = \sigma_e(N') = \sigma(N_0) = \sigma_e(N_0)$.  Whence $N'$ and $N_0$ are approximately unitarily equivalent in $\B(\H)$ and the result follows.
\end{proof}
Theorem \ref{lem:approx-SH-BH-essential} immediately implies the following result.
\begin{cor}
\label{cor:three-point-spectrum-SH-BH}
If $T \in \D$, then there exists a normal operator $N \in \B(\H)$ such that $\sigma(N) =\sigma_e(N) = \{\alpha_1, \alpha_2, \alpha_3\}$ where $\{\alpha_j\}^3_{j=1} \subseteq \bC$ are non-collinear points and $T \in \overline{\D_\D(N)}^{\left\|\cdot\right\|}$.
\end{cor}

\begin{example}
In \cite{A2007}*{Proposition 5} it was shown that there is no normal operator whose spectrum is $\{0,1,i\}$ and whose diagonal is
\[
T = \diag\left(\frac{1}{2}, \frac{i}{2}, \frac{1+i}{2}, 0, i, 1, 0, i, 1, 0, i, 1, \ldots\right).
\]
However, Corollary \ref{cor:three-point-spectrum-SH-BH} says $T$ is almost the diagonal of a normal operator with spectrum and essential spectrum $\{0,1,i\}$.
\end{example}

Using the above techniques and Lemma \ref{lem:approx-SH-BH-essential}, we obtain the following preliminary result to Theorem \ref{thm:SH-BH-Discrete-Continuous} .
\begin{prop}
\label{prop:approx-SH-BH-convex-essential}
Let $N$ be a normal operator in $\B(\H)$.  Then
\[
\{T \in \D \, \mid \, \sigma(T)  \subseteq \cconv(\sigma_e(N))\} \subseteq \overline{\D_\D(N)}^{\left\|\cdot\right\|}.
\]
\end{prop}
\begin{proof}
Fix $T \in \D$ such that $\sigma(T) \subseteq \cconv(\sigma_e(N))$ and let $\epsilon > 0$.  Choose $T_0 \in \D$ such that $\left\|T - T_0\right\| < \epsilon$, $\sigma_e(T_0) = \{\beta_k\}^n_{k=1}$, and $\sigma(T) \setminus \sigma_e(T) = \{\alpha_j\}^m_{j=1}$ where $\alpha_j,\beta_k \in \cconv(\sigma_e(N))$ (where $m = 0$ is possible).  

Since $N$ is a normal operator, there exists a normal operator $N_0 \in \B(\H)$ such that $\left\|N - N_0\right\| < \epsilon$, $\sigma_e(N_0) = \{\gamma_k\}^\ell_{k=1}$, $\sigma(N_0) \setminus \sigma_e(N_0) = \{\lambda_j\}^d_{j=1}$ (where $d = 0$ is possible), and $\{\beta_k\}^n_{k=1} \cup \{\alpha_j\}^m_{j=1} \subseteq \cconv(\{\gamma_k\}^\ell_{k=1})$.  To complete the proof, it suffices to show
\[
T_0 \in \overline{\D_\D(N_0)}^{\left\|\cdot\right\|}.
\]
Note that the case $d =0$ follows directly from Lemma \ref{lem:approx-SH-BH-essential}.  

For $j = 1,\ldots, m$ let $t_j$ be the multiplicity of the eigenvalue $\alpha_j$ in $T_0$.  Similarly, for each $j = 1,\ldots, d$ let $n_j$ be the multiplicity of the eigenvalue $\lambda_j$ in $N_0$.  By rearranging the orthonormal basis for $\D$, $T_0$ is unitarily equivalent to
\[
\left( \bigoplus^m_{j=1} \left( \alpha_j \oplus \beta_1 I_{\H_0}    \right)^{\oplus t_j} \right) \oplus \left(\bigoplus^n_{k=1} \beta_k I_{\H_0}\right)^{\oplus \sum^d_{j=1} n_j}
\]
where $\H_0$ is an infinite dimensional Hilbert space.  Similarly, $N_0$ is unitarily equivalent to
\[
\left( \bigoplus^d_{j=1} \left(\lambda_j \oplus \left(\bigoplus^n_{k=1} \gamma_k I_{\H_0}\right) \right)^{\oplus n_j} \right) \oplus \left(\bigoplus^\ell_{k=1} \gamma_k I_{\H_0}\right)^{\oplus \sum^m_{j=1} t_j}.
\]
Hence, if we can show that there exist unitaries $U_j$ such that the diagonal of
\[
U_j^*\left(\bigoplus^\ell_{k=1} \gamma_k I_{\H_0}\right)U_j
\]
is within $\epsilon$ of $\alpha_j \oplus \beta_1 I_{\H_0}$ for any $j$, and that there exist unitaries $W_k$ such that the diagonal of
\[
W_k^* \left(\lambda_j \oplus \left(\bigoplus^\ell_{k=1} \gamma_k I_{\H_0}\right)\right) W_k
\]
is within $\epsilon$ of $\oplus^n_{k=1} \beta_k I_{\H_0}$ for any k, then the proof will be complete. 

The $U_j$'s exist by Lemma \ref{lem:approx-SH-BH-essential}. To prove the existence of the $W_k$'s, note that 
\[
\lambda_j \oplus \left(\bigoplus^\ell_{k=1} \gamma_k I_{\H_0}\right)
\]
is unitarily equivalent to
\[
\left(\lambda_j \oplus \left(\bigoplus^\ell_{k=1} \gamma_k I_{\H_0}\right)\right) \oplus \left(\bigoplus^\ell_{k=1} \gamma_k I_{\H_0}\right)^{\oplus (n-1)},
\]
and that
\[
\bigoplus^n_{k=1} \beta_k I_{\H_0} = \beta_1 I_{\H_0} \oplus \left(\bigoplus^n_{k=2} \beta_k I_{\H_0}\right).
\]
Since Lemma \ref{lem:approx-SH-BH-essential} implies there is a unitary operator $V$ such that the diagonal of 
\[
V^*\left(\bigoplus^\ell_{k=1} \gamma_k I_{\H_0}\right)^{n-1}V
\]
is within $\epsilon$ of $\oplus^n_{k=2} \beta_k I_{\H_0}$, it suffices to prove the result for $n = 1$.
 
Since $\beta_1$ is in the convex hull of $\{\gamma_k\}^\ell_{k=1}$, there exist non-zero rational numbers $\{q_k\}^\ell_{k=1} \subseteq [0,1]$ such that $\sum^\ell_{k=1} q_k = 1$ and
\[
\left|\beta_1 - \sum^\ell_{k=1} q_k\gamma_k\right| < \epsilon.
\]
Choose $N \in \bN$ and $\{a_k\}^\ell_{k=1} \subseteq \bN \cup \{0\}$ such that $q_k = \frac{a_k}{N}$, and choose $M \in \bN$ such that
\[
\left|\beta_1 - \frac{\lambda_j + M \sum^\ell_{k=1} a_k\gamma_k}{NM+1}\right| < \epsilon.
\]

Let $D_1 \in \M_N(\bC)$ be the diagonal matrix where exactly $a_k$ of the diagonal entries of $D_1$ are $\gamma_k$ for each $k$. Let $D_2 \in \M_{NM+1}(\bC)$ be the diagonal matrix where $Ma_k$ of the diagonal entries are $\gamma_k$ for each $k$, and the remaining one diagonal entry is $\lambda_j$.  Then up to unitary equivalence,
\[
\lambda_j \oplus \left(\bigoplus^\ell_{k=1} \gamma_k I_{\H_0}\right) = D_2 \oplus D_1^{\oplus \infty}.
\]
Now by Theorem \ref{thm:matrix-constant-diagonal}, $D_1$ is unitarily equivalent to a matrix with every diagonal entry equal to 
\[
\frac{1}{N}\sum^\ell_{k=1} a_k \gamma_k = \sum^\ell_{k=1} q_k \gamma_k,
\]
which is within $\epsilon$ of $\beta_1$. Similarly, $D_2$ is unitarily equivalent to a matrix with every diagonal entry equal to
\[
\frac{\lambda_j + \sum^\ell_{k=1} Ma_k\gamma_k}{NM+1},
\]
which is also within $\epsilon$ of $\beta_1$.  The result now follows by taking a direct sum of unitary matrices obtained as above.
\end{proof}
\begin{thm}
\label{thm:SH-BH-discrete}
Theorem \ref{thm:SH-BH-Discrete-Continuous} holds in the case when $\A$ is a discrete MASA.
\end{thm}
\begin{proof}
Let $\mathfrak{K}$ denote the set of compact operators. If we identify $\D = \ell^\infty(\bN)$, then we can write $c_0 = \D \cap \mathfrak{K}$. Since $E_\D(K) \in c_0$ for every $K \in \mathfrak{K}$, it follows that the map $\tilde{E}_\D : \B(\H)/\mathfrak{K} \to \D/c_0$ defined by
\[
\tilde{E}_\D(T + \mathfrak{K}) = E_\D(T) + c_0.
\] is well defined and completely positive. Thus one inclusion follows from Lemma \ref{lem:spectrum-of-expectation-in-convex-hull}, and the other inclusion follows from Proposition \ref{prop:approx-SH-BH-convex-essential}.
\end{proof}

\subsubsection{The case of a continuous MASA}

In this section, using ideas similar to those used in the previous section, we will establish Theorem \ref{thm:SH-BH-Discrete-Continuous} for the case of a continuous MASA.  First, we require the following result from \cite{KS1959}*{Remark 5}.
\begin{lem}
\label{lem:continuous-masa-expectation-anhilates-compact}
Let $\A$ be a continuous MASA of $\B(\H)$ with corresponding conditional expectation $E_\A : \B(\H) \to \A$.  If $K \in \B(H)$ is compact, then $E_\A(K) = 0$.
\end{lem}

\begin{thm}
Theorem \ref{thm:SH-BH-Discrete-Continuous} holds in the case when $\A$ is a continuous MASA.
\end{thm}
\begin{proof}
To see that
\[
\overline{\D_\A(N)}^{\left\|\cdot\right\|} \subseteq \{A\in\A\mid\sigma(A)\subseteq\cconv(\sigma_e(N))\},
\]
let $\mathfrak{K} \subseteq \B(\H)$ and define the completely positive map $\tilde{E}_\A : \B(H)/\mathfrak{K} \to \A$ as in the proof of Theorem \ref{thm:SH-BH-discrete}. The inclusion then follows directly from Lemma \ref{lem:spectrum-of-expectation-in-convex-hull}.

For the other inclusion, fix $A\in\A$ such that $\sigma(A)\subseteq\cconv(\sigma(N))$
and let $\epsilon>0$.  A standard approximation argument gives the existence of $A_0 \in \A$ and $N_0 \in \B(\H)$ such that $\left\|A - A_0\right\| < \epsilon$, $\left\|N - N_0\right\| < \epsilon$, $\sigma(A_0)=\{\alpha_{k}\}^m_{k=1}$,
$\sigma_e(N_0)=\{\beta_{k}\}^n_{k=1}$, $\sigma(N_0)\setminus \sigma_e(N_0) = \{\lambda_{k}\}^\ell_{k=1}$, and $\sigma(A_0) \subseteq \cconv (\sigma(N_0))$.  It suffices to prove the result for $A_0$ and $N_0$.

Observe that there are non-zero pairwise orthogonal projections $\{P_k\}^m_{k=1}\subseteq\A$ such that
$\sum_{k=1}^{m}P_{k}=I_\H$ and $A_0 =\sum_{k=1}^{m}\alpha_{k}P_{k}$.   Furthermore, since $\sigma(A_0)\subset\cconv(\sigma(N_0))$, there are
scalars $\gamma_{j,k} \in (0,1)$ such that $\sum^n_{k=1} \gamma_{j,k} = 1$ and
\[
\left|\alpha_{j} -\sum^m_{k=1} \gamma_{j,k}\beta_{k}\right| < \epsilon
\]
for all $j$. 

For each $k$, let $A_{k}=\sum_{j=1}^{m}\gamma_{j,k}P_{j}$. Then 
\[
\left\|A_0-\sum_{k=1}^{n}\beta_{k}A_{k}\right\| < \epsilon,
\]
and $\{A_k\}^n_{k=1}$ is a partition of unity of $\A$.
By Theorem \ref{thm:continuous-masa-approx-carpenter} there
are pairwise orthogonal projections $\{Q_k\}^n_{k=1} \subseteq \B(\H)$, 
each of infinite rank, such that $\sum_{k=1}^{n}Q_{k}=I_\H$, $\sigma(Q_{k})=\sigma_{e}(Q_{k})$
and $\|E_{\A}(Q_{k})-A_{k}\|<\epsilon/C$ for all $k$,
where $C=\sum_{k=1}^{n}|\beta_{k}|$.

Let $\{Q'_j\}^\ell_{j=1}$ be pairwise orthogonal, finite-rank projections such that $Q'_j \leq Q_1$ and $\mathrm{rank}(Q'_j)$ is the multiplicity of the eigenvalue $\lambda_j$ of $N$ for all $j$. Replace $Q_1$ with $Q_1 - \sum^\ell_{j=1} Q'_j$. Then $Q_{1},\ldots,Q_{n}, Q'_1, \ldots, Q'_\ell \in\B(H)$ are pairwise orthogonal projections. Moreover, $E_\A(Q'_k) = 0$ for each $k$ by Lemma \ref{lem:continuous-masa-expectation-anhilates-compact}, and hence $\left\|E_\A(Q_k) - A_k\right\| < \frac{\epsilon}{C}$ for each $k$.

Let $N'=\sum_{k=1}^{n}\beta_{k}Q_{k} + \sum^\ell_{j=1} \lambda_j Q'_j$. Then
\[
\|E_{\A}(N')-A_0\|\leq\epsilon + \sum_{k=1}^{n}|\beta_{k}|\|E_{\A}(Q_{k})-A_{k}\|< 2\epsilon.
\]
Since $N'$ is a normal normal operator such that $\sigma(N')=\sigma(N_0)$ with equal multiplicity of finite eigenvalues, $N'$ is approximately unitarily
equivalent to $N_0$.  Hence the result follows.
\end{proof}

\subsection{Type II$_1$ factors}

In this section, we proceed with a Schur-Horn type theorem for normal operators in II$_1$ factors.  The discussion varies slightly from the I$_\infty$ case due to the existence of a faithful tracial state.
Denote by $\chi_{X}(N)$ the spectral projection of a normal operator $N$ with respect to a Borel set $X$.

The next result completely characterizes the approximate diagonal of a normal operator with finite spectrum in a II$_1$ factor.

\begin{thm}
\label{thm:schur-horn-II-1}
Let $\fM$ be a type II$_1$ factor, let $\tau$ be the faithful normal trace on $\fM$, and let $\A$ be a MASA of $\fM$ with corresponding normal conditional expectation $E_\A : \fM \to \A$.  Let $N \in \fM$ be a normal operator such that $\sigma(N) = \{z_k\}^n_{k=1} \subseteq \bC$.  Then
\[
A \in \overline{\D_\A(N)}^{\|\cdot\|}
\]
if and only if there exists a partition of unity $\{A_k\}^n_{k=1}$ in $\A$ such that $\tau(A_k) = \tau(\chi_{\{z_k\}}(N))$ for each $k$, and
\[
\sum^n_{k=1} z_k A_k = A.
\]
\end{thm}
\begin{proof}
Let $\omega_k = \tau(\chi_{\{z_k\}}(N))$ for each $k$. Suppose first that $A \in \D_\A(N)$, so that $A = E_\A(U^*NU)$ for some unitary $U \in \fM$.   For each $k$ let $P_k = \chi_{\{z_k\}}(N)$ and let $A_k = E_A(U^*P_kU) \in \A$.  Clearly $0 \leq A_k \leq I_\fM$, 
\[
\tau(A_k) = \tau(E_\A(U^*P_kU)) = \tau(P_k) = \omega_k \text{ for all }k,
\]
\[
\sum^n_{k=1} A_k = E_\A\left(U^*\left(\sum^n_{k=1} \chi_{\{z_k\}}(N)\right)U\right) = E_\A(I_\fM) = I_\fM,
\]
and 
\[
\sum^n_{k=1} z_k A_k= E_\A\left(U^*\left(\sum^n_{k=1} z_k\chi_{\{z_k\}}(N)\right)U\right) = E_\A(N) = A.
\]

Now suppose $A \in \overline{\D_\A(N)}^{\|\cdot\|}$.  Then there exists a sequence $(A^{(m)})_{m\geq 1} \subseteq \D_\A(N)$ such that $\lim_{m\to \infty} \left\|A - A^{(m)}\right\| = 0$.  For each $m \in \bN$ choose a partition of unity $\left\{A^{(m)}_k\right\}^n_{k=1}$ in $\A$ such that $\tau\left(A^{(m)}_k\right) = \omega_k$ for all $k$, 
and
\[
\sum^n_{k=1} z_k A^{(m)}_k = A^{(m)}.
\]
Since the unit ball of $\A$ is weak$^*$-compact, the result follows by taking a weak$^*$-limit over a common subnet since $\tau$ and $E_\A$ are weak$^*$-continuous.
 
For the converse direction, fix $A \in \A$ and suppose there exists a partition of unity $\{A_k\}^n_{k=1}$ in $\A$ such that  $\tau(A_k) = \omega_k$ for all $k$ and 
\[
\sum^n_{k=1} z_k A_k = A.
\]
Fix $\epsilon > 0$.  By Theorem \ref{thm:multivariate-carpenter} there exists a unitary operator $U \in \fM$ such that
\[
\left\|A_k - E_\A(U^*\chi_{\{z_k\}}(N)U)\right\| < \epsilon
\]
for all $k$.  Therefore
\[
\left\|A - E_\A(U^*NU)\right\| \leq \sum^n_{k=1} \left\|z_k (A_k -  E_\A(U^*\chi_{\{z_k\}}(N)U))\right\| \leq \epsilon \sum^n_{k=1} |z_k|.
\]
Therefore, since the $\{z_k\}^n_{k=1}$ are fixed, the result now follows.
\end{proof}
Although Theorem \ref{thm:schur-horn-II-1} only applies to normal operators with finite spectrum, it effectively solves the approximate Schur-Horn problem for normal operators via finite approximations.  For example, Theorem \ref{thm:schur-horn-II-1} implies the following result about arbitrary normal operators.
\begin{thm}
\label{thm:trace-II-1}
Let $\fM$ be a type II$_1$ factor, let $\tau$ be the faithful normal tracial state on $\fM$, and let $\A$ be a MASA.  If $N \in \fM$ is a normal operator, then
\[
\tau(N)I_\fM \in \overline{\D_\A(N)}^{\|\cdot\|}.
\]
\end{thm}
\begin{proof}
Since $\fM$ is a von Neumann algebra, for each $\epsilon > 0$ there exists pairwise orthogonal projections $\{P_k\}^n_{k=1}\subseteq \fM$ and scalars $\{z_k\}^n_{k=1} \subseteq \bC$ such that $\sum^n_{k=1} P_k = I_\fM$ and 
\[
\left\|N - \sum^n_{k=1} z_k P_k\right\| <\epsilon.
\]
Moreover, since $\tau$ is norm-continuous, we may assume, by applying a small perturbation, that
\[
\tau\left( \sum^n_{k=1} z_k P_k  \right) = \tau(N).
\]
 
Let $N_0 = \sum^n_{k=1} z_k P_k$ so that $N_0 \in \fM$ is a normal operator with $\sigma(N_0) = \{z_k\}^n_{k=1}$ and $\tau(\chi_{\{z_k\}}(N_0)) = \tau(P_k)$ for all $k$.  If $A_k = \tau(P_k)I_\fM$ for all $k$, then clearly $0 \leq A_k \leq I_\fM$, $\tau(A_k) = \tau(P_k)$, and 
\[
\sum^n_{k=1} A_k = \sum^n_{k=1} \tau(P_k) I_\fM = I_\fM.
\]
Thus 
\[
\tau(N_0) I_\fM = \sum^n_{k=1} z_k A_k \in \overline{\D_\A(N_0)}^{\left\|\cdot\right\|}
\]
by Theorem \ref{thm:schur-horn-II-1}.  Since $\left\|N - N_0\right\| < \epsilon$, the result follows.
\end{proof}

We are now able to prove Theorem \ref{thm:three-point-spectrum-SH-II-1}.
\begin{thm}
Theorem \ref{thm:three-point-spectrum-SH-II-1} holds. 
\end{thm}
\begin{proof}
One inclusion follows directly from Corollary \ref{cor:approx-diagonals-have-spectrum-in-convex-hull} and the fact that $E_\A$ is a trace-preserving, completely positive map.

For the converse, note that the conclusions of the theorem are unchanged if we replace $N$ by a scalar translation or a non-zero scalar multiple.  Thus we may assume $\sigma(N) = \{0, 1, z\}$ where $Im(z) > 0$.   The spectral distribution of $N$ may then be written as $\beta_0 \delta_0 + \beta_1 \delta_1 + \beta_z \delta_z$, where $\delta_x$ represents the point-mass measure at $x$ and $\beta_0, \beta_1, \beta_z \in (0,1)$ are such that $\beta_0 +\beta_1 + \beta_z = 1$.
 
Fix $A \in \mathcal{A}$ such that $\tau(A) = \tau(N)$ and $\sigma(A) \subseteq \mbox{conv}(\sigma(N))$.  Let $\epsilon > 0$.  Since $\mathfrak{M}$ is a type II$_1$ factor, there exists pairwise orthogonal projections $\{P_j\}^n_{j=1}\subseteq \mathcal{A}$ and scalars $\{\alpha_j\}^n_{j=1} \subseteq \mathbb{C}$ such that $\sum^n_{j=1} P_j = I_\mathfrak{M}$ and 
\[
\left\|A - \sum^n_{j=1} \alpha_j P_j\right\| < \epsilon.
\]
Furthermore, by refining the approximation, we may assume in addition that $\{\alpha_j\}^n_{j=1} \subseteq \mbox{conv}(\sigma(N))$ and
\[
\tau\left( \sum^n_{j=1} \alpha_j P_j  \right) = \tau(N).
\]

Since $\alpha_j \in \mbox{conv}(\sigma(N))$ for each $j$, there exists scalars $\gamma_{j,0}, \gamma_{j,1}, \gamma_{j,z} \in [0,1]$ such that $\gamma_{j,0}+ \gamma_{j,1}+ \gamma_{j,z} = 1$ and $\alpha_j = \gamma_{j,0} 0 + \gamma_{j,1} 1  +\gamma_{j,z} z$.  Thus
\[
\beta_0 0 + \beta_1 1 + \beta_z z = \tau(N) = \sum^n_{j=1} (\gamma_{j,0} 0 + \gamma_{j,1} 1  +\gamma_{j,z} z)\tau(P_j).
\]
Since $Im(z) > 0$, the above equation implies that
\[
\beta_z = \sum^n_{j=1} \gamma_{j,z}\tau(P_j).
\]
By analyzing the real part, we also obtain that
\[
\beta_1 = \sum^n_{j=1} \gamma_{j,1} \tau(P_j).
\]
Furthermore, since $\gamma_{j,0}+ \gamma_{j,1}+ \gamma_{j,z} = 1$ for each $j$ and since $\beta_0 +\beta_1 + \beta_z = 1$, we obtain that
\[
\beta_0 = \sum^n_{j=1} \gamma_{j,0} \tau(P_j).
\]

Based on the above equations and the fact that $\mathcal{A}$ is a MASA in $\mathfrak{M}$, for each $j$ there exists pairwise orthogonal projections $\{P_{j,0}, P_{j,1}, P_{j,z}\} \subseteq P_j\mathcal{A}P_j$ such that $P_j = P_{j,0} + P_{j,1} + P_{j,z}$,
\[
\tau(P_{j,0}) = \gamma_{j,0}\tau(P_j), \quad \tau(P_{j,1}) = \gamma_{j,1}\tau(P_j), \qand \tau(P_{j,z}) = \gamma_{j,z} \tau(P_j).
\]
Therefore, if 
\[
N' = \sum^n_{j=1} 0 P_{j,0} + 1 P_{j,1} + z P_{j,z},
\]
then $N' \in \mathcal{A}$ is a normal operator with the same spectral distribution as $N$.  Thus, as $\sigma(N)$ is finite, there exists a unitary operator $U \in \mathfrak{M}$ such that $N' = U^*NU$.
 
Notice that
\[
P_j N'P_j =0 P_{j,0} + 1 P_{j,1} + z P_{j,z}
\]
and
\[
\tau(0 P_{j,0} + 1 P_{j,1} + z P_{j,z}) = \gamma_{j,0}\tau(P_j) 0 + \gamma_{j,1}\tau(P_j) + \gamma_{j,z}\tau(P_j) = \alpha_j \tau(P_j).
\]
Hence, by applying Theorem \ref{thm:trace-II-1} to the type II$_1$ factor $P_j\mathfrak{M}P_j$, the MASA $P_j\mathcal{A}P_j$, and the normal operator $P_j N'P_j \in P_j\mathfrak{M}P_j$, there exists a unitary $W_j \in P_j\mathfrak{M}P_j$ such that 
\[
\left\|E_{P_j\mathcal{A}P_j}(W_j^*P_j N'P_jW_j) - \alpha_jP_j\right\| < \epsilon.
\]
Therefore $W = \sum^n_{j=1} W_j \in \mathfrak{M}$ is a unitary operator such that
\[
\begin{array}{l}
\left\|E_\mathcal{A}(W^*N'W)- \sum^n_{j=1} \alpha_j P_j\right\| \\ 
=  \left\|\sum^n_{j=1} E_{P_j\mathcal{A}P_j}(W_j^*P_j N'P_jW_j)- \sum^n_{j=1} \alpha_j P_j\right\|\\
 = \max_{1\leq j \leq n}\{\left\|E_{P_j\mathcal{A}P_j}(W_j^*P_j N'P_jW_j)- \alpha_j P_j\right\| \} < \epsilon.
 \end{array} 
\]
This gives
\begin{align*}
& \left\|E_\mathcal{A}(W^*U^*NUW) - A\right\| \\
&  \leq \left\|E_\mathcal{A}(W^*N'W)- \sum^n_{j=1} \alpha_j P_j\right\| + \left\|A- \sum^n_{j=1} \alpha_j P_j\right\| < 2\epsilon.
\end{align*}
Since $UW \in \mathfrak{M}$ is a unitary operator, the result follows.
\end{proof}

We immediately obtain the following corollary.
\begin{cor}
Let $\mathfrak{M}$ be a type II$_1$ factor, let $\tau$ be the faithful tracial state on $\mathfrak{M}$, let $\mathcal{A}$ be a MASA in $\mathfrak{M}$ with corresponding normal conditional expectation $E_\mathcal{A} : \mathfrak{M} \to \mathcal{A}$, and let $A \in \mathcal{A}$.  Then there exists a normal operator $N \in \mathfrak{M}$ such that $\sigma(N)$ contains precisely three points and
\[
A \in \overline{\D_\A(N)}^{\|\cdot\|}.
\]
\end{cor}

\begin{example}
\label{ex:four-noncollinear-points}
Theorem \ref{thm:three-point-spectrum-SH-II-1} may not be improved to normal operators with four points in their spectrum.  Indeed consider the normal operator $N$ whose spectral distribution is $\frac{1}{4}(\delta_0 + \delta_1 + \delta_i + \delta_{1+i})$ and the normal operator $A$ in the MASA $\mathcal{A}$ whose spectral distribution is $\frac{1}{2}(\delta_0 + \delta_{1+i})$.  It is clear that $\sigma(A) \subseteq \mbox{conv}(\sigma(N))$ and $\tau(A) = \tau(N)$.

If $A \in \overline{\D_\A(N)}^{\|\cdot\|}$, then there are subprojections $P_0, P_1, P_i, P_{i+1}$ of $\chi_{\{0\}}(N)$, $\chi_{\{1\}}(N)$, $\chi_{\{i\}}(N)$, and $\chi_{\{1+i\}}(N)$ respectively such that 
\[
\tau(P_0) + \tau(P_1) + \tau(P_i) + \tau(P_{i+1}) = \frac{1}{2}
\]
and
\[
0\tau(P_0) + 1\tau(P_1) + i\tau(P_i) + (1+i)\tau(P_{i+1}) = 0.
\]
However, this is clearly impossible, since $0$ is an extreme point of the square with vertices $\{0,1,1+i, i\}$, and the trace of the spectral projection of $N$ corresponding to $0$ is $\frac{1}{4}$.
 
Note that the pair $A$, $N$ do not satisfy the assumptions of Theorem \ref{thm:schur-horn-II-1}.
\end{example}

\subsection{Type II$_\infty$ factors}

The following analogue of Theorem \ref{thm:schur-horn-II-1} for II$_\infty$ factors can be proved in the same way using the version of Theorem \ref{thm:multivariate-carpenter} for II$_\infty$ factors.
\begin{thm}
Let $\fM$ be a type II$_\infty$ factor with separable predual, let $\tau$ be the semifinite normal trace on $\fM$, and let $\A$ be a MASA of $\fM$ with corresponding normal conditional expectation $E_\A : \fM \to \A$.  Let $N \in \fM$ be a normal operator such that $\sigma(N) = \{z_k\}^n_{k=1}$. Then
\[
A \in \overline{D_\A(N)}^{\|\cdot\|}
\]
if and only if there exists a partition of unity $\{A_k\}^n_{k=1}$ of $\A$ such that $\tau(A_k) = \tau(\chi_{\{z_k\}}(N))$ for all $k$ and
\[
\sum^n_{k=1} z_k A_k = A.
\]
\end{thm}

\subsection{Type III factors}

Unlike factors of type I$_\infty$ and II, the set of approximate diagonals of a normal operator in a factor of type $III$ is easily described.

\begin{thm}
\label{thm:schur-horn-III}
Let $\fM$ be a type III factor, let $\A$ be a MASA of $\fM$ with corresponding conditional expectation $E_\A : \fM \to \A$, and let $N \in \fM$ be a normal operator.  Then
\[
A \in \overline{\D_\A(N)}^{\|\cdot\|} \text{ if and only if }\sigma(A) \subseteq \cconv(\sigma(N)).
\]
\end{thm}
\begin{proof}
If $A \in \overline{\D_\A(N)}^{\|\cdot\|}$ then $\sigma(A) \subseteq \cconv(\sigma(N))$ by Corollary \ref{cor:approx-diagonals-have-spectrum-in-convex-hull}.
 
For the converse direction, fix $A \in \A$ and $N \in \fM$ such that $\sigma(A) \subseteq \cconv(\sigma(N))$.  Let $\epsilon > 0$.  Since $\A$ is a MASA in the type III factor $\fM$, there exists a normal operator $A_0 \in \A$ such that $\sigma(A_0)$ is a finite set, $\sigma(A_0) \subseteq \sigma(A) \subseteq \cconv(\sigma(N))$, and $\left\|A_0 - A\right\| < \epsilon$.  Furthermore, since $\sigma(A_0)$ is a finite set and $\fM$ is a type III factor, there exists a normal operator $N_0 \in \fM$ such that $\sigma(N_0)$ is finite, $\left\|N_0 - N\right\| < \epsilon$, and $\sigma(A_0) \subseteq \cconv(\sigma(N_0))$.
 
Let $\{\alpha_j\}^\ell_{j=1} = \sigma(A_0)$ and let $\{\beta_k\}^n_{k=1} = \sigma(N_0)$.  For each $j = 1\,\ldots, \ell$ let $P_j = \chi_{\{\alpha_j\}}(A_0)$.  Then $\{P_j\}^\ell_{j=1}$ is a collection of pairwise orthogonal, non-zero projections that sum to $I_\fM$.  Since $\fM$ is a type III factor, for each $j$ there exists a collection of pairwise orthogonal, non-zero projections $\{P_{j,k} \, \mid \, k  = 1,\ldots, n\} \subseteq \A$ such that $\sum^n_{k=1} P_{j,k} = P_j$.  Define
\[
N' = \sum^\ell_{j=1} \sum^n_{k=1} \beta_j P_{j,k}.
\]
The operator $N'$ is clearly normal and $\sigma(N') = \sigma(N_0)$.  Thus there exists a unitary $U \in \fM$ such that $U^*N_0U = N'$.
 
Fix $j \in \{1,\ldots, \ell\}$.  By construction, $N'$ commutes with $P_j$ and $\{\beta_k\}^n_{k=1} = \sigma(N_0)$ is the spectrum of $P_jN'P_j$ inside of $P_j \fM P_j$.  Since $\alpha \in \cconv(\sigma(N_0)) = \cconv(\{\beta_k\}^n_{k=1})$, there exists $\{\gamma_k\}^n_{k=1} \subseteq [0,1]$ such that $\sum^n_{k=1} \gamma_k = 1$ and 
\[
\sum^n_{k=1} \gamma_k \beta_k = \alpha_j.
\]
Hence by Theorem \ref{thm:approx-carpenter-III} implies there exists a unitary $W_j \in P_j\fM P_j$ such that
\[
\left\|\gamma_k P_j - E_{P_j\A P_j}(W_j^*P_{j,k}W_j)\right\| < \frac{\epsilon}{n(M+1)},
\]
where  $M = \max_k|\beta_k|$.  Hence
\begin{align*}
\big\|\alpha_j P_j - E_{P_j\A P_j}(W_j^* P_j&N'P_jW_j)\big\| \\
&\leq \sum^n_{j=1} \big\|\gamma_k \beta_k P_j - \beta_k E_{P_j\A P_j}(W_j^*P_{j,k}W_j)\big\| \\
& \leq \sum^n_{j=1} |\beta_k| \frac{\epsilon}{n(M+1)}  \\
& < \epsilon.
\end{align*}
 
By applying the above construction for each $j$, we obtain that $W = \sum^\ell_{j=1} W_j \in \fM$ is a unitary operator satisfying
\begin{align*}
\left\|E_\A(W^*N'W)- A_0\right\| &=  \left\|\sum^\ell_{j=1} E_{P_j\A P_j}(W_j^*P_j N'P_jW_j)- \alpha_j P_j\right\|\\
 &= \max_{1\leq j \leq \ell}\left\|E_{P_j\A P_j}(W_j^*P_j N'P_jW_j)- \alpha_j P_j\right\| 
 \leq \epsilon.\\
\end{align*}
Thus $\left\|E_\A(W^*U^*N_0UW) - A_0\right\|\leq \epsilon$, so
\[
\left\|E_\A(W^*U^*NUW) - A\right\|\leq 3\epsilon.
\]
Since $UW \in \fM$ is a unitary operator, the result follows.
\end{proof}

\subsection{Cuntz C*-algebras}

It follows from Remark \ref{rem:UHF-no-SH} that, in general, one cannot hope to have any kind of tracial condition in a multivariate Carpenter's theorem for UHF C$^*$-algebras. This precludes us from applying the techniques used in previous sections to arbitrary UHF C$^*$-algebras in order to obtain a Schur-Horn type theorem for normal operators.

Therefore, we restrict our attention to a setting where the trace should not matter, namely the Cuntz algebra $\mathcal{O}_2$.  It is well-known that there exists a copy of the $2^\infty$-UHF C$^*$-algebra $\mathcal{F}_2$ in $\mathcal{O}_2$, and a conditional expectation $E_{\mathcal{F}_2} : \mathcal{O}_2 \to \mathcal{F}_2$.  In particular, if $E : \mathcal{F}_2 \to \mathfrak{D}$ is the expectation onto the diagonal, then the map $E_\mathfrak{D} : \mathcal{O}_2 \to \mathfrak{D}$ defined by $E_\mathfrak{D} = E \circ E_{\mathcal{F}_2}$ is a conditional expectation onto the diagonal.  For this expectation, we obtain the following result.

\begin{thm}
\label{thm:approx-carpenter-O2}
Let $E_\mathfrak{D} : \mathcal{O}_2 \to \mathfrak{D}$ be the conditional expectation described above and let $N \in \mathcal{O}_2$ be a normal operator such that $\sigma(N) = \{\alpha_k\}^n_{k=1}$.  Then
\[
\overline{\D_\mathfrak{D}(N)}^{\left\|\cdot\right\|} = \overline{\left\{ \left. \sum^n_{k=1} \alpha_k A_k \, \right| \, \{A_k\}^n_{k=1} \text{ a partition of unity in } \mathfrak{D}\right\}  }^{\left\|\cdot\right\|}.
\]
\end{thm}
\begin{proof}
If $A =E_\mathfrak{D}(U^*NU)$ for some $U \in \U(\mathcal{O}_2)$, then it is clear that
\[
A \in \left\{ \left. \sum^n_{k=1} \alpha_k A_k \, \right| \, \{A_k\}^n_{k=1} \text{ a partition of unity in } \mathfrak{D}\right\},
\] 
since the image of the spectral projections of $U^*NU$ under $E_\mathfrak{D}$ form a partition of unity of $\mathfrak{D}$. This gives one direction.

For the other direction, suppose that $A=\sum^n_{k=1} \alpha_k A_k$ for some partition of unity $\{A_k\}^n_{k=1} \subseteq \mathfrak{D}$. Theorem \ref{thm:UHF-MVCT} implies there exists non-zero orthogonal projections $\{P_k\}^n_{k=1} \subseteq \mathcal{F}_2 \subseteq \mathcal{O}_2$ such that
\[
\left\|E_\mathfrak{D}(P_k) - A_k\right\| < \epsilon.
\]
In particular,
\[
\left\|A - E_\mathfrak{D}\left(\sum^n_{k=1} \alpha_k P_k\right)\right\| < M\epsilon,
\]
where $M = \sum^n_{k=1} |\alpha_k|$.  The operator $\sum^n_{k=1} \alpha_k P_k$ is normal with the same spectrum as $N$. Hence it is unitarily equivalent to $N$ by $K$-theory.  The result now follows.
\end{proof}
\begin{cor}
Let $E_\mathfrak{D} : \mathcal{O}_2 \to \mathfrak{D}$ be the conditional expectation described above and let $S \in \mathcal{O}_2$ be a self-adjoint operator.  Then
\[
\overline{\D_\mathfrak{D}(S)}^{\left\|\cdot\right\|} = \{A \in \mathfrak{D} \, \mid \, \sigma(A) \subseteq \cconv(\sigma(S))\}.
\]
\end{cor}
\begin{proof}
The inclusion 
\[
\overline{\D_\mathfrak{D}(S)}^{\left\|\cdot\right\|} \subseteq \{A \in \mathfrak{D} \, \mid \, \sigma(A) \subseteq \cconv(\sigma(S))\}
\]
follows from Corollary \ref{cor:approx-diagonals-have-spectrum-in-convex-hull}.

For the other inclusion, we may assume without loss of generality that $\sigma(S)$ is finite, since $\mathcal{O}_2$ has real rank zero.  Write $\sigma(S) = \{\alpha_k\}^n_{k=1}$ with $\alpha_1 < \alpha_2 < \cdots < \alpha_n$, and fix $A \in \mathfrak{D}$ such that $\sigma(A) \subseteq \cconv(\sigma(S))$.  Since $\O_2$ has real rank zero, we may further assume that $\sigma(A)$ is finite, and that $\alpha_1, \alpha_n \notin \sigma(A)$.  Write $\sigma(A) = \{\beta_j\}^m_{j=1}$ and $A = \sum^m_{j=1} \beta_j Q_j$, where $\{Q_j\}^m_{j=1}$ are spectral projections.  

Fix $\epsilon > 0$.  Since $\alpha_1 < \beta_j < \alpha_n$ for each $j$, there exists non-zero scalars $\{\gamma_{k,j}\}^n_{k=1} \subseteq (0,1)$ such that $\sum^n_{k=1} \gamma_{k,j} = 1$ and $\left|\sum^n_{k=1} \alpha_k \gamma_{k,j} - \beta_j\right| < \epsilon$. 

Let $A_k = \sum^m_{j=1} \gamma_{k,j} Q_j \in \mathfrak{D}$.  Then since $\{A_k\}^n_{k=1}$ is a partition of unity in $\mathfrak{D}$, by Theorem \ref{thm:approx-carpenter-O2} there exists non-zero pairwise orthogonal projections $\{P_k\}^n_{k=1} \subseteq \mathcal{F}_2 \subseteq \mathfrak{D}$ such that $\left\|E_\mathfrak{D}(P_k) - A_k\right\| < \epsilon$.  

Setting $S_0 = \sum^n_{k=1} \alpha_k P_k$ and $M = \sum^n_{k=1} |\alpha_k|$, we estimate
\[
\left\|E_\mathfrak{D}(S_0) - \sum^n_{k=1} \alpha_k A_k \right\| = \left\| \sum^n_{k=1} \alpha_k E_\mathfrak{D}(P_k) - \sum^n_{k=1} \alpha_k A_k\right\| \leq \epsilon M
\]
and
\[
\left\| A - \sum^n_{k=1} \alpha_k A_k\right\| = \left\| \sum^m_{j=1} \beta_j Q_j - \sum^m_{j=1} \sum^n_{k=1} \alpha_k \gamma_{k,j} Q_j\right\| < \epsilon
\]
Since $S_0$ and $S$ are unitarily equivalent in $\mathcal{O}_2$, the result now follows.
\end{proof}



\begin{thebibliography}{99}




\bib{AM2007}{article}{
    author={Argerami, M.},
    author={Massey, P.},
    title={A Schur-Horn Theorem in II$_1$ Factors},
    journal={Indiana Univ. Math. J.},
    volume={56},
    number={5},
    date={2007},
    pages={2051-2060}
}



\bib{AM2009}{article}{
    author={Argerami, M.},
    author={Massey, P.},
    title={Towards the Carpenter’s Theorem},
    journal={Proc. Amer. Math. Soc.},
    volume={137},
    number={11},
    date={2009},
    pages={3679-3687}
}




\bib{A2007}{article}{
    author={Arveson, W.},
    title={Diagonals of normal operators with finite spectrum},
    journal={Proc. Natl. Acad. Sci. USA},
    volume={104},
    number={4},
    date={2007},
    pages={1152-1158}
}


\bib{AK2006}{book}{
    author={Arveson, W.},
    author={Kadison, V.},
    title={Diagonals of self-adjoint operators},
    series={Operator Theory, Operator Algebras, and Applications},
	publisher={Amer. Math. Soc.},
    volume={414},
    date={2006},
    place={Providence, RI},
    pages={247-263}
}





\bib{AP1979}{article}{
    author={Au-Yeung, Y.H.},
    author={Poon, Y.T.},
    title={$3 \times 3$ orthostochastic matrices and the convexity of generalized numerical ranges},
    journal={Linear Algebra Appl.},
    volume={27},
    year={1979},
    pages={69-79}
}






\bib{BR2014}{article}{
    author={Bhat, B. V.},
    author={Ravichandran, M.},
    title={The Schur-Horn theorem for operators with finite spectrum},
    journal={Proc. Amer. Math. Soc.},
    volume={142},
    number={10},
    date={2014},
    pages={3441-3454}
}




\bib{BJ2013}{article}{
    author={Bownik, M.},
    author={Jasper, J.},
    title={The Schur-Horn theorem for operators with finite spectrum},
    eprint={arXiv:1302.4757},
    year={2013},
    pages={10}
}

\bib{BJ2014}{article}{
    author={Bownik, M.},
    author={Jasper, J.},
    title={Spectra of frame operators with prescribed frame norms},
    journal={Proceedings of the 9th International Conference on Harmonic Analysis and Partial Differential Equations, Contemp. Math.},
    volume={612},
    year={2014},
    pages={65-79}
}




\bib{DFHS2012}{article}{
    author={Dykema, K.},
    author={Fang, J.},
    author={Hadwin, D.},
    author={Smith, R.},
    title={The Carpenter and Schur-Horn Problems for MASAs in Finite Factors},
    journal={Illinois J. Math.},
    number={4},
    volume={56},
    date={2012},
    pages={1313-1329}
}









\bib{H1954}{article}{
    author={Horn, A.},
    title={Doubly stochastic matrices and the diagonal of a rotation matrix},
    journal={Amer. J. Math.},
    volume={76},
    number={3},
    date={1954},
    pages={620-630}
}



\bib{J2013}{article}{
    author={Jasper, J.},
    title={The Schur-Horn theorem for operators with three point spectrum},
    journal={J. Funct. Anal.},
    volume={265},
    number={8},
    date={2013},
    pages={1494-1521}
}


\bib{K2002-1}{article}{
    author={Kadison, R.},
    title={The Pythagorean theorem. I. The finite case},
    journal={Proc. Natl. Acad. Sci. USA},
    volume={99},
    number={7},
    date={2002},
    pages={4178-4184}
}

\bib{K2002-2}{article}{
    author={Kadison, R.},
    title={The Pythagorean theorem. II. The infinite discrete case},
    journal={Proc. Natl. Acad. Sci. USA},
    volume={99},
    number={8},
    date={2002},
    pages={5217-5222}
}



\bib{KW2008}{article}{
    author={Kaftal, V.},
    author={Weiss, G.},
    title={A survey on the interplay between arithmetic mean ideals, traces, lattices of operator ideals, and an infinite Schur-Horn majorization theorem},
    journal={Hot topics in operator theory, Theta Ser. Adv. Math.},
    volume={9},
    date={2008},
    pages={101-135}
}



\bib{KW2010}{article}{
    author={Kaftal, V.},
    author={Weiss, G.},
    title={An infinite dimensional Schur-Horn theorem and majorization theory},
    journal={J. Funct. Anal.},
    volume={259},
    number={12},
    date={2010},
    pages={3115-3162}
}


\bib{KS1959}{article}{
    author={Kadison, R.},
    author={Singer, I.},
    title={Extensions of pure states},
	journal={Amer. J. Math},
    date={1959},
    pages={383-400}
}

\bib{KS2014}{article}{
    author={Kennedy, M.},
    author={Skoufranis, P.},
    title={Thompson's theorem for II$_1$ factors},
    eprint={arXiv:1407.1564},
    date={2014},
    pages={20}
}



\bib{MR2014}{article}{
    author={Massey, P.},
    author={Ravichandran, M.},
    title={Multivariable Schur-Horn theorems},
    eprint={arXiv:1411.4457 },
    date={2014},
    pages={31}
}


\bib{M1964}{article}{
    author={Mirsky, L.},
    title={Inequalities and existence theorems in the theory of matrices},
    journal={J. Math. Anal. Appl.},
    volume={9},
    number={1},
    date={1964},
    pages={99-118}
}







\bib{N1999}{article}{
    author={Neumann, A.},
    title={An Infinite-Dimensional Version of the Schur-Horn Convexity Theorem},
    journal={J. Funct. Anal.},
    volume={161},
    number={2},
    date={1999},
    pages={418-451}
}





\bib{P1981}{article}{
    author={Popa, S.},
    title={On a problem of RV Kadison on maximal abelian $^*$-subalgebras in factors},
    journal={Invent. Math.},
    volume={65},
    number={2},
    date={1981},
    pages={269-281}
}



\bib{P1983}{article}{
    author={Popa, S.},
    title={Orthogonal Pairs of $^*$-Subalgebras in Finite von Neumann Algebras},
    journal={J. Oper. Theory},
    volume={9},
    date={1983},
    pages={253-268}
}











\bib{R2012}{article}{
    author={Ravichandran, M.},
    title={The Schur-Horn theorem in von Neumann algebras},
    eprint={arXiv:1209.0909v2},
    date={2012},
    pages={26}
}


\bib{S1923}{article}{
    author={Schur, I.},
    title={\"{U}ber eine Klasse von Mittelbildungen mit Anwendungen auf die Determinantentheorie},
    journal={Sitzungsber. Berl. Math. Ges.},
    volume={22},
    date={1923},
    pages={9-20}
}








\bib{T1977}{article}{
    author={Thompson, R. C.},
    title={Singular values, diagonal elements, and convexity},
    journal={SIAM J APPL MATH},
    volume={32},
    number={1},
    date={1977},
    pages={39-63}
}




\end{thebibliography}
\end{document}